\documentclass{conm-p-l}
\usepackage{}
\usepackage[all]{xy}
\usepackage{amssymb,amsthm, amsmath, amsfonts}

\newtheorem{theorem}{Theorem}[section]
\newtheorem{corollary}[theorem]{Corollary}
\newtheorem{lemma}[theorem]{Lemma}
\newtheorem{proposition}[theorem]{Proposition}

\theoremstyle{definition}

\theoremstyle{remark}
\newtheorem{remark}[theorem]{Remark}

\numberwithin{equation}{section}

\newcommand\Ker{\, {\rm Ker}\, }
\newcommand\supp{{\rm supp}}
\newcommand\dist{{\rm dist}}
\newcommand\spec{{\rm spec}}

\begin{document}

\title{Semiclassical Asymptotics on Manifolds with Boundary}

\author{Nilufer Koldan}
\address{Department of Mathematics, Northeastern University, Boston, MA, USA
} \email{koldan.n@neu.edu}
\thanks{The first author was partially supported by NSF grant
DMS-0600196 and NSF grant DMS-0400426.}

\author{Igor Prokhorenkov}
\address{Department of Mathematics, Texas Christian University, Fort Worth, TX, USA}
\email{i.prokhorenkov@tcu.edu}

%    Information for third author
\author{Mikhail Shubin}
\address{Department of Mathematics, Northeastern University, Boston, MA, USA
} \email{shubin@neu.edu}
\thanks{The third author was partially supported by NSF grant DMS-0600196.}

\date{}


\begin{abstract}
%Old version
%%In this paper,
%We have found semiclassical asymptotics for the
%eigenvalues of the Witten Laplacian
%%up to any fixed constant
%for compact manifolds with boundary in the presence of a general
%Riemannian metric. To this end, we have modified and used the
%variational method suggested by Kordyukov, Mathai and Shubin (2005),
%% \cite{KMS}
%with a more extended use
%of quadratic forms instead of the operators. We also utilize  some
%important ideas and technical elements from Helffer and Nier (2006).
%\cite{HN},
%%as well as the technique of model operators, as formulated in
%%\cite{S1}.
%%Shubin (1996).
%
%New version
%In this paper
We discuss semiclassical asymptotics for the
eigenvalues of the Witten Laplacian
%up to any fixed constant
for compact manifolds with boundary in the presence of a general
Riemannian metric. To this end, we modify and use the
variational method suggested by Kordyukov, Mathai and Shubin (2005),
% \cite{KMS}
with a more extended use
of quadratic forms instead of the operators. We also utilize  some
important ideas and technical elements from Helffer and Nier (2006),
who were the first to supply a complete proof of the full semi-classical
asymptotic expansions for the eigenvalues with fixed numbers.
%\cite{HN},
%as well as the technique of model operators, as formulated in
%\cite{S1}.
%Shubin (1996).

\end{abstract}

\maketitle

\section{Introduction}\label{S:intro}

{\bf A.} In his famous paper \cite{W}, E.~Witten introduced a deformation
of the de Rham complex of differential forms on a compact closed
manifold $M$.
%was introduced
%by E.\ Witten \cite{W} in his famous paper, published in 1982.
It is a new (``deformed") complex which depends upon a given Morse
function $f$  on $M$ and contains a small parameter $h>0$
(``Planck's constant"). The deformed differential is given by the
formula
\begin{equation*}
 d_{h,f}\omega = h e^{-f/h} d
(e^{f/h} \omega) = hd \omega + df \wedge \omega,
\end{equation*}
where $\omega$ is an exterior differential form on $M$, $d_{h,f}^2=0$.
Choosing a Riemannian metric $g$ on $M$, we can take
the corresponding normalized deformed Laplacian

\begin{equation*}
\Delta _{h,f,g} \omega
=h^{-1}(d_{h,f}^* d_{h,f} + d_{h,f} d^*_{h,f})
= h \Delta \omega  +
\left(\mathcal{L}_{\nabla f} + \mathcal{L}_{\nabla f}^{*}\right)\omega
+h^{-1}|\nabla f|^2 \omega,
\end{equation*}
where $\mathcal{L}_{\nabla f}$ is the Lie derivative along $\nabla f$,
the adjoint operators $d^*_{h,f}$,  ${\mathcal L}^*_{\nabla f}$
(to  $d_{h,f}$,  ${\mathcal L}_{\nabla f}$ respectively)
are taken with respect to the scalar products defined by the metric $g$
(and by the corresponding smooth measure on $M$) on the exterior forms on $M$;
$\Delta=d^*d+dd^*$ is the usual Laplacian on forms. The deformed Laplacian
$\Delta_{h,f,g}$ is often called the Witten Laplacian.\\

Multiplication  by $\sqrt{h}\,e^{f/h}$ defines an isomorphism
between the deformed complex with the differential $d_{h,f}$ and
the standard de Rham complex (with the differential $d$). In
particular, the cohomology spaces of these complexes are
isomorphic. By the Hodge theory, these cohomology spaces are
naturally isomorphic to the corresponding spaces of harmonic
forms, i.e.\ the  kernels (null-spaces) of the Laplacians. It
follows that
\begin{equation*}
\dim\Ker \Delta_{h,f,g}^{(p)}=\dim\Ker \Delta^{(p)} = b_p(M),
\end{equation*}
where  $\Delta_{h,f,g}^{(p)}$, $\Delta^{(p)}$ denote the
restrictions of the corresponding Laplacians to $p$-forms,
$b_p(M)$ is the $p$th real Betti number of $M$.\\

An  important feature of the Laplacian $\Delta_{h,f,g}$ is as follows:
for small $h$ the eigenforms corresponding to
the bounded eigenvalues, are small outside a small neighborhood of
the critical points of $f$ because the potential $V= |\nabla f|^2$
does not vanish there, and if eigenforms do not localize around the
critical points, then the term $h^{-1} |\nabla f|^2$ will be larger than
the sum of all other terms in the equation for the eigenfunction,
provided $h$ is sufficiently small. Therefore, we can expect that
only small neighborhoods of the critical points play a role
in semiclassical asymptotics of the eigenvalues; in particular,
we can expect that only ``principal parts" of $f$ and  $g$
are relevant. For example,  we can hope that only quadratic parts of $f$
and constant (flat) metrics at every critical point contribute to the principal term
in the semiclassical asymptotics of the eigenvalues (i.e. asymptotics as $h \rightarrow 0$).\\

Based on this idea, Witten gave an analytic proof of the Morse
inequalities on compact smooth manifolds without boundary. In
their simplest form (see \cite{M}), these inequalities state that
the number $m_p$ of the critical points with index $p$ of a Morse
function $f$ can not be less than the Betti number $b_p$ of the
underlying manifold: $m_p\ge b_p$, for all $p$. Semiclassical
asymptotics of the eigenvalues relate these two numbers by
including both of them into one mathematical object: the Witten
deformation of the de Rham complex, where $m_p$ becomes the number
of small eigenvalues (multiplicity counted) of
$\Delta_{h,f,g}^{(p)}$ and $b_p$ is the multiplicity of the $0$ as
the eigenvalue of the same deformed
Laplacian on $p$-forms. This immediately implies the Morse inequalities above.\\

Rigorous versions of Witten's proof, with additional attention to
details related to the quantum tunneling, appeared in
papers by  B.~Simon \cite{S} (see also the  book \cite{CFKS}),
B.~Helffer and J.~Sj\"{o}strand \cite{HS} and others.\\

{\bf B.} The definition of Morse function $f$ extends to manifolds with
boundary if in addition we assume that $f$ has no critical points
on the boundary and the restriction of $f$ to the boundary is also
Morse. In its simplest form the Morse inequalities state that the
number of critical points of index $p$ of $f$ plus the number of
critical points of index $p-1$ (resp. $p$) of $f|_{\partial M}$
with positive (resp. negative) outward normal derivative is not
smaller than the $p$-th relative (resp. absolute) Betti number of
the underlying manifold. A topological proof of this fact was
obtained by E. Baiada and M. Morse in 1953 in \cite {BM}. For a
modern topological treatment and
generalizations to manifolds with corners see \cite{F}.\\

On manifolds with boundary the Witten Laplacian is defined by the
same formula as in the case without boundary, but now we need to
specify its domain. To obtain a differential complex, it is
natural to choose the domain of the Witten differential $d_{h,f}$
as consisting of the forms with vanishing tangential (or normal)
components on the boundary. This defines the quadratic form of the
corresponding Witten Laplacian, and we will mainly consider
Witthen Laplacian as the operator, defined by the closed quadratic
form. The domain of the Laplacian requires additional
vanishing conditions on the adjoint of the Witten differential (\cite{HN}).\\

K. C. Chang and J. Liu \cite{CL} were the first to use the method
of the Witten Laplacian to give an analytic proof of Morse
inequalities for compact manifolds $M$ with boundary by
considering semiclassical asymptotics of small eigenvalues for the
Witten Laplacian. Following the ideas in \cite{CFKS}, Chang and
Liu only had to study the case when the metric $g$ and the Morse
function $f$ have canonical flat forms near the critical points.
(This is sufficient to prove the Morse inequalities as
a statement in differential topology.)\\

In 2006, B.~Helffer and F.~Nier \cite{HN} found semiclassical
asymptotics of the Witten Laplacian on compact manifolds with boundary
with the general Riemannian metric.
They were mainly interested in
obtaining very accurate asymptotics for the first (exponentially
small) eigenvalue on functions.
% (other eigenvalues were not considered).
 A new feature which appears here is an influence on
the asymptotics of the behavior of the Morse function $f$ near some
critical points of its restriction to the boundary. In particular,
B.~Helffer and F.~Nier had to study the ``rough" localization of
the spectrum of the Witten Laplacian on forms.
\\

In the present paper, in contrast to \cite{CL} and \cite{HN}, we
give a new proof  of the
%prove
semiclassical asymptotics for every  eigenvalue of the  Witten
Laplacian with a fixed number (in increasing order) for compact
manifolds with boundary in the presence of a general Riemannian
metric. To this end, we modify a method suggested in \cite{KMS}
(where a similar result with some applications, including a
vanishing result for the Quantum Hall conductivity, was obtained
on regular coverings of compact manifolds without boundary). We
will use some important technical elements from Helffer and Nier,
as well as the technique of model operators, as formulated in \cite{S1}.\\

The purpose of the present paper is to provide  a new method
of establishing semi-classical asymptotics of
any eigenvalue  of Witten's Laplacian in the case of a smooth
compact manifold with smooth boundary. We consider the boundary
conditions obtained by choosing a domain for $d_{h,f}$. Namely, we
take  the domain of the corresponding quadratic form to consist of
all forms of appropriate smoothness which have vanishing
tangential (resp.\ normal) parts on the boundary. In this case
eigenforms with bounded eigenvalues localize around the interior
critical points and only those boundary critical points, i.e.
critical points of $f|_{\partial M}$, which have a positive (resp.
negative) outward normal derivative of $f$. In the spirit  of
\cite{S1} and \cite{S2}, we construct the model operator which is
the direct sum of two parts: one corresponding to the interior
critical points and the other to the boundary critical points, as
specified above. The part of the model operator corresponding to
the interior critical points is the same as for manifolds without
boundary. Namely, we choose coordinates such that the critical
point $x_i$ is the origin. In these coordinates $\Delta_{h,f,g}$
can be written as
\begin{equation*}
\Delta_{h,f,g}=-h A + B +h^{-1}V(x).
\end{equation*}
For each $x_i$ we obtain the model operator on
$H^2(\mathbb{R}^{n}; \Lambda T^*\mathbb{R}^{n})$ by replacing $A$
with its highest order terms with the coefficients frozen at the
critical point, $B$ with $B(x_i)$ and $V$ with its quadratic part
near the critical point. Then we take the direct sum of these
operators over all interior critical points. At the relevant
boundary critical points we construct the model operator in the
same way, but this time for the Witten Laplacian on $\partial M$
with the function $f|_{\partial M}$ and the constant metric $g'$
which is obtained by restricting $g$ to the vectors tangent to
$\partial M$ and then freezing it at the critical point. We prove
that the spectrum of Witten Laplacian that is below an arbitrarily
chosen constant $R$ concentrates around the part of the spectrum
of the model operator that is below the same constant $R$ as $h
\rightarrow 0$. Here we use techniques from  \cite{KMS}. In the
proof, the part corresponding to the interior critical points is
the same as the one in case when there is no boundary. The part
corresponding to the boundary critical points  is harder to treat,
and we use appropriately modified ideas of Helffer
and Nier \cite{HN}.\\

{\bf C.}  In the last 25 years the method of Witten deformation was
successfully applied to prove a number of significant results in
topology and analysis. We provide a very brief review of
literature. We note that our choices are highly subjective and are
influenced by our own interests.\\

In 1982  J.-M. Bismut \cite{B} modified the Witten deformation
technique and combined it with intricate and deep probabilistic
methods to produce a new proof of the degenerate Morse-Bott
inequalities (see \cite{Bo} for topological proof). A more
accessible proof based on the adiabatic technique of
Mazzeo-Melrose and Forman (\cite{MM}, \cite{Fo1}) was given by I.
Prokhorenkov in \cite{P} (see also \cite{HS2} for a different
approaches to the proof).\\

A. V. Pazhitnov \cite{Pa} used the method of Witten deformation to
prove some of the Morse-Novikov inequalities  when the gradient of
Morse function is replaced by a closed 1-form. Novikov
inequalities for vector fields were established by M. A. Shubin in
\cite{S2}. Shubin's results were extended by M. Braverman and M.
Farber \cite{BF1} to the case when 1-form (or corresponding
vector field) has non-isolated zeros, and to the equivariant case in \cite{BF2}.\\

J. Alvarez L\'{o}pez \cite{AL} used the method of Witten to prove
Morse inequalities for the invariant cohomology of the space of
orbits with applications to basic cohomology of Riemann
foliations. V. Belfi, E. Park, and K. Richardson \cite{BPR} used
the Witten deformation of the basic Laplacian to prove an analog
of Hopf index theorem for Riemannian foliations. Further
applications of the method of Witten deformation to index theory
were developed in \cite{PR1} and \cite{PR2}.\\

The method of Witten Laplacian was also used to study the analytic
torsion of the Witten complex. The analytic torsion was introduced
by D. B. Ray and I. M. Singer \cite{RaS}. For odd dimensional
manifolds, D. B. Ray and I. M. Singer conjectured that the
analytic torsion and the Reidemeister torsion coincide.
Independently, J. Cheeger \cite{C} and W. M\"{u}ller \cite{Mu}
have proved this conjecture. The methods of J. Cheeger and W.
M\"{u}ller are both based on a combination of topological and
analytical methods. Then J.M. Bismut and W. Zhang \cite{BZ}
suggested a purely analytical proof of the Cheeger-M\"{u}ller
theorem and generalized it to the case where the metric is not
flat. Later another analytic proof was suggested by D. Burghelea,
L. Friedlander and T. Kappeler \cite{BFK} which was shorter but
based on application of the highly non-trivial Mayer-Vietoris type
formula for the determinant of an elliptic operator. In this paper
they generalize the theorem to the case of manifolds of any
dimension (not necessarily odd).  M. Braverman \cite{BR}
found a short analytic proof by a direct way of analyzing the
behaviour of the determinant of the Witten deformation of the Laplacian.\\

Finally, M. Braverman and V. Silantiev used the method of Witten
deformation to extend Novikov Morse-type inequalities for closed
1-forms $\omega$ to manifolds with boundary in \cite{BS}. In the
paper they require that the form $\omega$ is exact near the
boundary of the manifold and that its critical set satisfies the
condition of F. C. Kirwan (see \cite{K}). The Witten deformation
technique then is used to obtain discrete spectrum and to localize
topological computations to the neighborhood of the critical set
of $\omega $.\\

\section{Preliminaries and the main theorem}\label{S:Pre}

Suppose that $M$ is a $C^{\infty}$ compact manifold,
$\dim_{\mathbb R} M =n$, with $C^{\infty}$ boundary $\partial M$,
and $g$ is a Riemanian metric on $M$. Let $f$ be a Morse function
on $M$, that is, $f$ is a Morse function on $M$ with no critical
points on the
boundary and $f\mid_{\partial M}$ is also a Morse function on $\partial M$.\\

We will use the following notations. The cotangent bundle on $M$
is denoted by $T^*M$, and the  bundle of exterior forms is
$\Lambda T^*M = \oplus_{k=0}^n \Lambda^k T^*M$. The spaces of smooth
sections of the bundles $\Lambda T^*M,\  \Lambda^k T^*M$ are denoted by
$\Lambda (M),\  \Lambda^k(M)$ respectively.
The elements of the sets $\Lambda (M)$ and $\Lambda^k(M)$ are called smooth
differential forms on $M$ and
smooth $k$-forms respectively. Given a manifold $M$ and a metric $g$ on $M$, we
will use the notations $(\cdot , \cdot)_g$ and $\parallel \cdot
\parallel _g$ for the $\textrm{L}^2$-inner product and the $\textrm{L}^2$-norm on $\Lambda (M)$
defined by the metric $g$.\\

It is convenient to use orientations of  $M$ and $\partial M$ (if the orientations exist)
to make global integrals well defined. If there is no orientation,
we should define the integrals as sums of them over small disjoint pieces,
each one located
over a coordinate neighborhood. It is easy to see that every term (hence, every sum)
does not depend upon the choice of orientations. For the sake of simplicity
of notations, we will
always assume both $M$ and $\partial M$ to be oriented.\\

Let $\bar n$ be the outward unit normal vector field defined on
the boundary of $M$ and $\bar n^*$ be its dual $1-$form with
respect to the metric $g$. For any $\omega \in \Lambda^k(M)$,
define the tangential part of $\omega$ as $\mathbf{t}(\omega)(y)=
i_{\bar n(y)}(\bar n^*(y) \wedge \omega(y))$ and the normal part
of $\omega$ as $\mathbf{n}(\omega)(y)= \omega(y) -
\mathbf{t}(\omega)(y)$ for any $y \in \partial M$. So
$\mathbf{t}\omega$ and $\mathbf{n} \omega$ are sections of
$\Lambda T^*M|_{\partial M}$. If $j:
\partial M \hookrightarrow M$ is the inclusion map then $j^*$
defines isomorphism between $\Lambda(\partial M)$ and
$\{\mathbf{t} (\omega): \omega \in \Lambda(M)\}$. By the Sobolev
trace theorem, $\mathbf{t}$ extends by continuity to a linear map
of Sobolev spaces
\begin{equation*}
\mathbf{t}: H^1(M; \Lambda T^*M)\rightarrow H^{1/2}(\partial M;
\Lambda T^*M) \subset L^2(\partial M; \Lambda T^*M),
\end{equation*}
and so does ${\bf n}$.\\

The Witten deformation of the exterior derivative is defined by
\begin{equation*}
    d_{h,f}=e^{-f/h}hde^{f/h}=hd+df \wedge,
\end{equation*}
where $h>0$ is a parameter
The adjoint of $d_{h,f}$ with respect to the
$\textrm{L}^2$-inner product $(\cdot , \cdot)_g$ is
\begin{equation*}
    d^{*}_{h,f,g}=e^{f/h}hd^{*}e^{-f/h}=hd^*+i_{\nabla f}.
\end{equation*}
Now define a quadratic form which
is the closure of
\begin{eqnarray}\label{E:MainQuad}
Q_{h,f,g}(\omega)&=&\frac{1}{h}\left(\parallel d_{h,f} \omega
\parallel_{g} ^2 +\parallel d^{*}_{h,f,g} \omega
\parallel_{g} ^2\right)
\end{eqnarray}
with the domain $\{\omega \in C^{\infty}(M; \Lambda T^*M):
\mathbf{t}(\omega)=0\}$. The same notation will be used for the
closure of this form. The closure is well defined and its domain
is
\begin{equation} \label{E:QDomain}
D(Q_{h,f,g})= \{\omega \in H^1(M; \Lambda T^*M) :
\mathbf{t}(\omega)=0\} \end{equation}
where $H^1$ denotes the Sobolev space of forms.\\

The quadratic form $Q_{h,f,g}$ can be written as
\begin{eqnarray}\label{E:MainQuad2}
Q_{h,f,g}(\omega)&=&h\left(\Delta \omega , \omega
\right)_{g}+\left((\mathcal{L}_{\nabla f} + \mathcal{L}_{\nabla
f}^{*}) \omega , \omega \right)_{g}+h^{-1}\left(|\nabla f|^2
\omega ,\omega \right)_{g}\\
 &&+ \int_{\partial M}(\mathbf{t} \bar\omega) \wedge (\star \mathbf{n}d_{h,f} \omega)
 - \int_{\partial M}(\mathbf{t} d^*_{h,f,g} \omega) \wedge (\star \mathbf{n} \bar\omega). \notag
\end{eqnarray}
(see (2.12) in \cite{HN}). In this formula $\mathcal{L}_{\nabla
f}$ and $\mathcal{L}^{*}_{\nabla f}$ denote the Lie derivative in
the direction of $\nabla f$ and its $\textrm{L}^2$-adjoint, and
$\star $ is the Hodge operator.\\

The Witten Laplacian is the elliptic self-adjoint operator associated with
the closed quadratic form $Q_{h,f,g}$ by the Friedrichs
construction (\cite{RS}, vol.1, section VIII.6). It is the
operator defined by
\begin{equation}\label{E:MainOp}
\Delta _{h,f,g} \omega= h \Delta \omega  +
\left(\mathcal{L}_{\nabla f} + \mathcal{L}_{\nabla f}^{*}\right)
\omega +h^{-1}|\nabla f|^2 \omega
\end{equation}
with the domain $D(\Delta _{h,f,g})= \{\omega \in H^2(M; \Lambda
T^*M) :~ \mathbf{t}(\omega)=0 ,~ \mathbf{t}(d^*_{h,f,g}
\omega) = 0 \}$ where $H^2$ is the Sobolev space of the corresponding sections.\\

On its domain, $Q_{h,f,g}$ can be also written as
%E:MainQuadlocal
\begin{eqnarray*}
Q_{h,f,g}(\omega) = &h\left(\parallel d \omega
\parallel_{g} ^2 +\parallel d^{*} \omega
\parallel_{g} ^2\right)+\left((\mathcal{L}_{\nabla f} +
\mathcal{L}_{\nabla f}^{*}) \omega , \omega \right)_{g}
+h^{-1}\left(|\nabla f|^2 \omega ,\omega \right)_{g}\\
 &-\displaystyle\int_{\partial M}\frac{\partial f}{\partial \bar n} (x) < \omega,
\omega >_g(x)~ d\mu_{\partial M}(x). \notag
\end{eqnarray*}
(see (2.15) in \cite{HN}) where $< \omega, \omega
>_g(x)$ is the inner product on $\Lambda T_x^*M$ induced by the metric $g$
(see p. 226-227 \cite{CFKS}) and $\mu_{\partial M}$ is the measure
on $\partial M$ defined by the metric $g$.\\

\textbf{Definition:} \textit{The set of critical points} is a
disjoint union of the sets $S_0$, $S_+$ and $S_-$ where $S_0$ is
the set of interior points which are critical points of the Morse
function $f$, $S_+$ and $S_-$ are the sets of boundary points
which are critical points of $f|_{\partial M}$ with
$\frac{\partial f}{\partial \bar n} > 0$ and $\frac{\partial
f}{\partial \bar n} < 0$ respectively. Let $S = S_0 \cup S_+$.\\

If $y \in \partial M$ then $y \in S_{\pm}$ if and only if $\frac{\partial
f}{\partial \bar n}(y) = \pm|\nabla f(y)|$ respectively.\\

Assume that there exist $N_0$ interior critical points denoted by
$\{\bar{x}_1, ... , \bar{x}_{N_0}\}$ and $N$ points in $S_+$
denoted by $\{\bar{y}_1,..., \bar{y}_N\}$. In Section
\ref{S:MainThmAndProof}, we will prove that
these points are the critical points we need to consider.\\

Now we want to form the model operator for the Witten Laplacian
$\Delta _{h,f,g}$ which provides the best approximation of $\Delta _{h,f,g}$
by a direct sum of harmonic oscillators near the critical
set $S_{0}$ .\\

For each critical point $\bar x_i \in S_0$, we define the operator $\Delta_i$  on
$H^2(\mathbb{R}^{n}; \Lambda T^*\mathbb{R}^{n})$ to be
\begin{equation*}
\Delta_i = -h A_{i} + B_{i} +h^{-1}V_{i}(x),
\end{equation*}
where the operator $-A_{i}$ is the principal part of the Laplacian
$\Delta$ at $\bar x_{i}$. It is an elliptic second order
differential operator with constant coefficients. Operator $B_{i}$
is the value at $\bar x_{i}$ of the bounded self-adjoint zero
order operator $\mathcal{L}_{\nabla f} + \mathcal{L}_{\nabla
f}^{*}$. Finally, $V_{i}(x)$ is the quadratic part of the
potential $|\nabla f|^{2}$ near $\bar x_{i}$. In local coordinates
$x_1, ... , x_n$ near $\bar x_i$, let $g_i=g(\bar x_i)$, then
\begin{equation*}
A_{i}=\sum_{l,k=1}^{n} g_i^{lk}\frac{\partial^2}{\partial x_l
\partial x_k}~,
\end{equation*}
\begin{equation*}
B_{i} = \sum_{k,r=1}^{n} \frac{\partial^2 f}{\partial
x_{r}\partial x_{k}}(\bar{x}_i)~ (a^{*k} a^{r}-a^r a^{*k})
\end{equation*} and
\begin{equation*}
V_{i}(x)=\sum_{l,k,r,s=1}^{n} g_i^{rs} \frac{\partial^2 f}{\partial
x_{r}\partial x_{l}}(\bar{x}_i) \frac{\partial^2 f}{\partial x_{s}
\partial x_{k}}(\bar{x}_i) x_l x_k
\end{equation*}
where $a^k=(dx_{k}\wedge)^{*} $ and $a^{*k}=(a^{k})^{*}=dx_{k}\wedge$ are the fermionic annihilation
and creation operators. Note that at the interior critical points, the model
operator is the same as in the model operator for manifolds without boundary (see \cite{S1}).\\

For each boundary critical point $\bar{y}_j \in S_+$,in a small neighborhood of $\bar{y}_j$ let
$f_j'=f\mid_{\partial M}$ and $g_j'$ be the metric obtained by
restricting $g$ on the tangential vectors and freezing it at the
critical point. Now define the operator $\Delta_j$ on
$H^2(\mathbb{R}^{n-1}; \Lambda T^*\mathbb{R}^{n-1})$ as
\begin{equation*}
\Delta_j = -h A_{j}' + B_{j}' +h^{-1}V_{j}'(x)
\end{equation*}
where we define operators $A_{j}'$, $B_{j}'$, and  $V_{j}'$ as
before by considering $\partial M$ as a manifold without boundary.
In local coordinates $x_1, ... , x_{n-1}$ for $\partial M$ near
$\bar y_j$ we have
\begin{equation*}
A_{j}'=\sum_{l,k=1}^{n-1} (g_j')^{lk}\frac{\partial^2}{\partial x_l
\partial x_k}~,
\end{equation*}
\begin{equation*}
B_{j}' = \sum_{k,r=1}^{n-1}  \frac{\partial^2
f_j'}{\partial x_{r}\partial x_{k}}(\bar{y}_j)~ (a^{*k} a^{r}-a^r
a^{*k})
\end{equation*} and
\begin{equation*}
V_{j}'(x)=\sum_{l,k,r,s=1}^{n-1} (g_j')^{rs} \frac{\partial^2
f_j'}{\partial x_{r}\partial x_{l}}(\bar{y}_j) \frac{\partial^2
f_j'}{\partial x_{s} \partial x_{k}}(\bar{y}_j) x_l x_k.
\end{equation*}\\

The \textit{model operator} is defined by
\begin{equation}\label{E:ModelOp}
\Delta_{\textrm{mod}} = \left( \oplus_{i=1}^{N_0} \Delta_{i} \right) \oplus
\left( \oplus_{j=1}^{N} \Delta_{j} \right).
\end{equation}
The model operator does not depend on the choice of local coordinates on $M$.\\

It follows from the general theory of elliptic operators on manifolds
with boundary that the operator $\Delta _{h,f,g}$ has discrete spectrum with
eigenvalues
\begin{equation*}
\lambda_1(h) \leq \lambda_2(h) \leq \lambda_3(h) \leq~...
\end{equation*}
such that $\lambda_l(h) \rightarrow \infty$ as $l \rightarrow
\infty$
for each $h>0$.\\

The spectrum of the model operator $\Delta_{\textrm{mod}}$ is also
discrete and the eigenvalues are independent of $h$ \cite{S1}. We list
all elements of the spectrum of $\Delta_{\textrm{mod}}$ in the
increasing order as
\begin{equation*}
\mu_1 \leq \mu_2 \leq \mu_3 \leq~...
\end{equation*}
such that $\mu_l \rightarrow \infty$ as $l \rightarrow \infty$.\\

We will prove that up to any fixed real number $R\notin$ ${\rm
spec}\,(\Delta_{\textrm{mod}})$, the spectrum of the operator
$\Delta_{h,f,g}$ concentrates near the spectrum of the model
operator $\Delta_{\textrm{mod}}$ as $h \rightarrow 0$. More
precisely, our main result is

\begin{theorem}\label{T:SpecClose} For every positive number $R
\notin$ ${\rm spec}\,(\Delta_{\rm{mod}})$ there exist $M>0$,
$h_0>0$ and $C>0$ such that both $\Delta_{\rm{mod}}$ and
$\Delta_{h,f,g}$ have exactly $M$ eigenvalues less than $R$ and
\begin{equation*}
|\lambda_l(h) - \mu_l| \leq C h^{\frac{1}{5}} ,\quad l=1,2,...,M,
\quad h\in (0, h_0).
\end{equation*}
\end{theorem}
\bigskip

\begin{remark} One can replace the tangential boundary conditions
\eqref{E:QDomain} for the quadratic form $Q_{h,f,g}$ with the normal
boundary conditions
\begin{equation}
D(Q_{h,f,g})= \{\omega \in H^1(M; \Lambda T^*M) :
\mathbf{n}(\omega)=0\}. \end{equation}
The corresponding domain for the Witten Laplacian $\Delta _{h,f,g}$
is  $D(\Delta _{h,f,g})= \{\omega \in H^2(M; \Lambda
T^*M) :~ \mathbf{n}(\omega)=0 ,~ \star \mathbf{n}(d_{h,f}
\omega) = 0 \}$. The case of the normal boundary conditions can be
reduced to one of tangential boundary conditions by observing that
the Hodge $\star$ operator maps the space $\{\omega \in H^2(M; \Lambda
T^*M) :~ \mathbf{n}(\omega)=0 ,~ \star \mathbf{n}(d_{h,f}
\omega) = 0 \}$ to the space $\{\omega \in H^2(M; \Lambda
T^*M) :~ \mathbf{t}(\omega)=0 ,~ \mathbf{t}(d^*_{h,(-f),g}
\omega) = 0 \}$, and that $\star \Delta _{h,f,g}=\Delta _{h,(-f),g}\star $.
\end{remark}

We will prove our main result by comparing the Witten Laplacian to the
model operator in a way suggested in Theorem 2.1 in \cite{KMS}.
We will give the abstract setting, which emphasizes the use of
quadratic forms instead of operators, in
the next section and proofs can be found in the appendix.\\

\section{General results on equivalence of
projections}\label{s:abstract-equiv}

Consider Hilbert spaces ${\mathcal H}_1$ and ${\mathcal H}_2$
equipped with inner products $(\cdot,\cdot)_1$ and
$(\cdot,\cdot)_2$. Let $Q_1$ and $Q_2$ be closed bounded
below quadratic forms with dense domains $D(Q_1) \subset
\mathcal{H}_1$ and $D(Q_2) \subset \mathcal{H}_2$ respectively.
Let $A_1$ and $A_2$ be the self-adjoint operators corresponding to
the quadratic forms.\\

Let us take $\lambda_{01}, \lambda_{02} \leq 0$ such that
\begin{equation}\label{e:6}
Q_l(\omega) \geq \lambda_{0l} \|\omega\|_l^2, \quad \omega \in
D(Q_l),\quad l=1,2.
\end{equation}\\

Let ${\mathcal H}_0$ be a Hilbert space, equipped with injective
bounded linear maps $i_1 : {\mathcal H}_0 \to {\mathcal H}_1$
and $i_2 : {\mathcal H}_0 \to {\mathcal H}_2$. Assume that there
are given bounded linear maps $p_1:{\mathcal H}_1\to {\mathcal
H}_0$ and $p_2:{\mathcal H}_2\to {\mathcal H}_0$ such that
$p_1\circ i_1={\rm id}_{{\mathcal H}_0}$ and $p_2\circ i_2={\rm
id}_{{\mathcal H}_0}$, as in the following diagram:

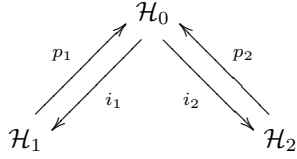
\begin{equation*}
\xymatrix @=4pc @ur { \mathcal H_1 \ar@<1ex>[r]^{p_1} & \mathcal
H_0\ar@<-1ex>[d]_{i_2} \ar@<1ex>[l]^{i_1} \\  & \mathcal H_2
\ar@<-1ex>[u]_{p_2} }
\end{equation*}

Let $J$ be a self-adjoint bounded operator  in ${\mathcal H}_0$.
Assume that $(i_2Jp_1)^*=i_1Jp_2$.\\

Since the operators $i_l:{\mathcal H}_0\to {\mathcal H}_l,~
l=1,2,$ are bounded and have bounded left-inverse operators $p_l$,
they are topological monomorphisms, i.e. they have closed images
and the maps $i_l:{\mathcal H}_0\to {\rm Im}\, i_l$ are
topological isomorphisms. Therefore, we can assume that the
estimate
\begin{equation}\label{e:rho}
\rho^{-1}\|i_2J\omega\|_2\leq \|i_1J\omega\|_1\leq
\rho\|i_2J\omega\|_2, \quad \omega\in {\mathcal H}_0,
\end{equation}
holds with some $\rho>1$. (Although we can choose the constant
$\rho$ in the estimate \eqref{e:rho} to be independent of $J$, it
may be possible to choose $\rho$ closer to 1, due to the presence of $J$.)\\

Define the bounded operators $J_l$ in ${\mathcal H}_l,~ l=1,2,$ by
the formula $J_l=i_lJp_l$. We assume that
\begin{itemize}
\item the operator $J_l,~ l=1,2,$ maps the domain of $Q_l$ to
itself;

\item $J_l$ is self-adjoint, and $0\leq J_l\leq {\rm
id}_{{\mathcal H}_l},~ l=1,2$;

\item for $\omega\in {\mathcal H}_0$, $i_1J\omega\in D(Q_1)$ if
and only if $i_2J\omega\in D(Q_2)$.
\end{itemize}
Denote $D=\{\omega\in {\mathcal H}_0 : i_1J\omega\in
D(Q_1)\}=\{\omega\in
{\mathcal H}_0 : i_2J\omega\in D(Q_2)\}.$\\

Introduce a self-adjoint positive bounded linear operator $J'_l$
in ${\mathcal H}_l$ by the formula $J_l^2+{J'_l}{}^2={\rm
id}_{{\mathcal H}_l}$. We assume that
\begin{itemize}
\item the operator $J'_l,~ l=1,2,$ maps the domain of $Q_l$ to
itself;

\item the quadratic forms $Q_l(\omega)-(Q_l(J_l \omega) + Q_l(J_l'
\omega))$ are bounded i.e.
\begin{equation}\label{e:15}
Q_l(J_l \omega) + Q_l(J_l' \omega)-Q_l(\omega) \leq \gamma_{l}
\|\omega\|_l^2, \quad \omega \in D(Q_l),\quad l=1,2.
\end{equation}
\end{itemize}

Finally, we assume that
\begin{equation}\label{e:14}
Q_l(J'_l\omega)\geq \alpha_l \|J'_l\omega\|_l^2,\quad \omega\in
D(Q_l), \quad l=1,2,
\end{equation}
for some $\alpha_l>0$, and
\begin{align}\label{e:16}
Q_2(i_2J\omega)\leq \beta_1 Q_1(i_1J\omega)
+\varepsilon_1\|i_1J\omega\|_1^2, \quad \omega\in
D,\\
\label{e:A1A2} Q_1(i_1J\omega)\leq \beta_2 Q_2(i_2J\omega)
+\varepsilon_2\|i_2J\omega\|_2^2, \quad \omega\in D,
\end{align}
for some $\beta_1,\beta_2\geq 1$ and $\varepsilon_1,
\varepsilon_2>0$.\\

Denote by $E_l(\lambda),~ l=1,2$, the spectral projection of the
operators $A_l$, corresponding to the semi-axis
$(-\infty,\lambda]$.\\

\begin{theorem}\label{t:equivalence}
Under the assumptions in this section, let $b_1>a_1$ and
\begin{align}\label{e:a2}
a_2&=\rho\left[ \beta_1 \left(a_1+\gamma_1+
\frac{(a_1+\gamma_1-\lambda_{01})^2}{\alpha_1-a_1-\gamma_1}\right)+
\varepsilon_1\right],\\ \label{e:b2}
b_2&=\frac{\beta_2^{-1}(b_1\rho^{-1}-\varepsilon_2)(\alpha_2-\gamma_2)
-\alpha_2\gamma_2+2\lambda_{02}\gamma_2-\lambda^2_{02}}
{\alpha_2-2\lambda_{02}+\beta_2^{-1}(b_1\rho^{-1}-\varepsilon_2)}.
\end{align}
Suppose that $\alpha_1>a_1+\gamma_1$, $\alpha_2>b_2+\gamma_2$ and
$b_2>a_2$. If the interval $(a_1,b_1)$ does not intersect with the
spectrum of $A_1$, then:

(1) the interval $(a_2,b_2)$ does not intersect with the spectrum
of $A_2$;

(2)The operator $E_2(\lambda_2)i_2Jp_1E_1(\lambda_1): {\rm Im}
E_1(\lambda_1)\to {\rm Im} E_2(\lambda_2)$ is an isomorphism for
any $\lambda_1 \in (a_1, b_1)$ and $\lambda_2 \in (a_2, b_2)$.
\end{theorem}
\bigskip
For the proof of this theorem see Appendix.\\

\begin{remark}\label{R: J_2} If  the spectral projections $E_l(\lambda),~ l=1,2$
have finite rank for all $\lambda $, then the condition
$(i_2Jp_1)^*=i_1Jp_2$ is not necessary and the condition that the
operator $J_2$ is self-adjoint can be replaced by a weaker
condition that $J_2$ is merely symmetric on ${\rm Im}(J_2)\subset
\mathcal{H}_2$. The projections $E_l(\lambda),~ l=1,2$ have finite
rank in the case of the Witten Laplacian $\Delta _{h,f,g}$.
\end{remark}

\begin{remark}\label{R:thesame} Since $\rho>1, \beta_1\geq 1, \gamma_1>0$ and
$\varepsilon_1>0$, we, clearly, have $a_2>a_1$. The formula
(\ref{e:b2}) is equivalent to the formula
\[
b_1=\rho\left[ \beta_2 \left(b_2+\gamma_2+
\frac{(b_2+\gamma_2-\lambda_{02})^2}{\alpha_2-b_2-\gamma_2}\right)+
\varepsilon_2\right],
\]
which is obtained from (\ref{e:a2}), if we replace $\alpha_1,
\beta_1, \gamma_1, \varepsilon_1, \lambda_{01}$ by $\alpha_2,
\beta_2, \gamma_2, \varepsilon_2, \lambda_{02}$ accordingly and
$a_1$ and $a_2$ by $b_2$ and $b_1$ accordingly. In particular,
this implies that $b_1>b_2$.
\end{remark}

\section{Proof of the main theorem}\label{S:MainThmAndProof}
We start by describing the setting for the application of Theorem
\ref{t:equivalence}.\\

Let $A_{2}$ be the Witten Laplacian $\Delta_{h,f,g}$ with the
domain
\begin{equation*}
D(\Delta _{h,f,g})= \{\omega \in H^2(M; \Lambda T^*M) :~
\mathbf{t}(\omega)=0 ,~ \mathbf{t}(d^*_{h,f,g} \omega) = 0 \}
\end{equation*}
(see \eqref{E:MainOp}). The operator $A_{2}$
corresponds to the quadratic form
\begin{equation*}
Q_{2}(\omega )=Q_{h,f,g}(\omega)=\frac{1}{h}\left(\parallel
d_{h,f} \omega
\parallel_{g} ^2 +\parallel d^{*}_{h,f,g} \omega
\parallel_{g} ^2\right)
\end{equation*}
with the domain $D(Q_2)= \{\omega \in H^1(M; \Lambda T^*M) :
\mathbf{t}(\omega)=0\}$ (see \eqref{E:MainQuad}).
Let $A_{1}=
\Delta_{\textrm{mod}}$ be the model operator (see \eqref{E:ModelOp}),
and $Q_{1}$ be
the quadratic form corresponding to the operator $A_1$ with the
domain
\begin{equation*}
    D(Q_1)= H^1(\mathbb{R}^n, \Lambda T^*\mathbb{R}^n)^{N_0}
\oplus H^1(\mathbb{R}^{n-1}, \Lambda T^*\mathbb{R}^{n-1})^{N}.
\end{equation*}
We have
\begin{equation*}
    D(Q_2)\subset \mathcal{H}_2 = L^2(M, \Lambda T^*M),
   \end{equation*}  and
\begin{equation*}
\begin{aligned}
    D(Q_1)\subset \mathcal{H}_1&=
L^2(\mathbb{R}^n, \Lambda T^*\mathbb{R}^n)^{N_0} \oplus
L^2(\mathbb{R}^{n-1}, \Lambda T^*\mathbb{R}^{n-1})^{N}\\ & \cong
\left(L^2(\mathbb{R}^n, \Lambda T^*\mathbb{R}^n) \otimes
\mathbb{C}^{N_0}\right) \oplus \left(L^2(\mathbb{R}^{n-1}, \Lambda
T^*\mathbb{R}^{n-1}) \otimes
\mathbb{C}^N \right).
\end{aligned}
\end{equation*}\\

For each interior critical point $\bar{x}_i \in S_0$, we choose local
coordinates $x_1, ..., x_n$. Let $B(\bar{x}_i, r) \subset M$ be
the open ball around $\bar{x}_i$ with radius $r$ and $B_i(0,r)$
be the corresponding ball in $\mathbb{R}^n$ in these coordinates.\\

Recall that at each boundary critical point $\bar{y}_j \in S_+$, we have
$\frac{\partial f}{\partial \bar n}(\bar{y}_j) = |\nabla
f(\bar{y}_j)|$. Then it is possible to find local coordinates
$x_1, x_2, ..., x_n$ near $\bar{y}_j$ such that in these
coordinates $\bar{y}_j$ is the origin, $\partial M = \{x_n=0\}$,
$M=\{x_n \leq 0\}$,
\begin{equation}\label{E:f}
f(x)= x_n + f'(x'),
\end{equation} and
\begin{equation}\label{E:g}
g=g_{nn}(x) dx_n^2 + g'(x),
\end{equation}
where $x =(x', x_n)$ (see (3.27) in \cite{HN} and Appendix B in \cite{F}). Here
$f'=f\mid_{\partial M}$, $g'$ is the restriction of the metric $g$
to the tangent space spanned by $\{{\partial /
\partial x_1}, {\partial / \partial x_2},...,{\partial / \partial
x_{n-1}}\}$ and $x'$ is any coordinates on $\partial M$ such that
$\bar{y}_j$ is the origin. Let $g_j=g(0)$ be the
constant metric in these coordinates. Furthermore, since $f\mid_{\partial
M}$ is a Morse function on the boundary,
the tangential coordinates $x_1, ..., x_{n-1}$ can be chosen so that
\begin{equation}\label{E:f'}
f'(x')= f(0)+\sum_{r=0}^{n-1} d_r x_r^2,
\end{equation}
where the coefficients $d_r$ for $r=1,...,n-1$ in the expression
of $f'$ are non vanishing real constants.
\\

Let $C(\bar{y}_j, r)= \{ x \in M: |x'| < r, ~-r <x_n \leq 0 \}$
for some $r>0$ and let $C_j(0,r)$ be the corresponding set in
$\mathbb{R}^n_{-}$. Choose $r$ small enough so that around each
boundary critical points $\bar{y}_j$ we can choose the special
coordinates, all the sets $B(\bar{x}_i, r)$ and $C(\bar{y}_j, r)$
are disjoint, and each $B(\bar{x}_i, r)$ is in the interior of
$M$. Let $B(\bar{y}_j, r) \subset \partial M$ be the open ball
around $\bar{y}_j$ with radius $r$ on the boundary $\partial M$,
and let $B_j(0,r)$ be the corresponding set in $\partial
\mathbb{R}^n_{-} = \mathbb{R}^{n-1}$. Let
\begin{equation*}
\mathcal{H}_0= (\oplus_{i=1}^{N_0} L^2(B_i(0, r), \Lambda
T^*\mathbb{R}^n|_{B_i(0, r)})) \oplus (\oplus_{j=1}^{N} L^2(B_j(0,
r), \Lambda T^*\mathbb{R}^{n-1}|_{B_j(0, r)})).
\end{equation*}

For some technical reasons that will become clear later, we choose
$\kappa$ such that $\frac{1}{3} < \kappa < \frac{1}{2}$. Let $\psi
~ \in C^\infty(\mathbb{R}^{n})$ such that $0 \leq \psi \leq 1$,
$\psi(x) = 1$ if $|x| \leq 1$, $\psi(x) = 0$ if $|x| \geq 2$. For
small enough $h$, $\psi_i^{(h)} (x)= \psi(h^{-\kappa}x) \in
C_c^\infty(B_i(0, r))$. Let $\phi ~ \in
C^\infty(\mathbb{R}^{n-1})$  such that $0 \leq \phi \leq 1$,
$\phi(x) = 1$ if $|x| \leq 1$, $\phi(x) = 0$ if $|x| \geq 2$. For
small enough $h$, $\phi_j^{(h)} (x)= \phi(h^{-\kappa} x) \in
C_c^\infty(B_j(0, r))$. Let $J$ be the multiplication operator by
$ (\oplus_{i=1}^{N_0} \psi_i^{(h)}(x)) \oplus
(\oplus_{j=1}^N \phi_j^{(h)}(x))$ in $\mathcal{H}_0$.\\

Let $i_{1} : \mathcal{H}_0 \rightarrow \mathcal{H}_1$ be the
natural inclusion  and let $p_{1}: \mathcal{H}_1 \rightarrow
\mathcal{H}_0$ be the restriction map, then $p_1\circ i_1={\rm
id}_{\mathcal{H}_0}$. Furthermore, the operators $J_1=i_1Jp_1$ and
$J_{1}^{'}$ clearly satisfy the five properties listed after the
definition of $J_1$ in Section \ref{s:abstract-equiv}. Indeed, the
last property follows from the calculation
\begin{eqnarray*}
Q_1(J_1 \omega) + Q_1(J_1' \omega)-Q_1(\omega) &=& (A_1J_1
\omega,J_1 \omega)_1+ (A_1J'_1 \omega,J'_1 \omega)_1-(A_1 \omega,
\omega)_1\\
&=& ((J_1A_1J_1 + J'_1A_1J'_1 -A_1) \omega,\omega)_1\\
&=& h\sum_{i=1}^{N_0}((| d\psi_i^{(h)}|^2  +
|d\tilde{\psi}_i^{(h)}|^2)\omega, \omega )_1\\
&&+h\sum_{j=1}^{N} ((|d\phi_j^{(h)}|^2 +
|d\tilde{\phi}_j^{(h)}|^2) \omega,\omega )_1,
\end{eqnarray*}
where $\tilde{\psi}_i^{(h)}= \sqrt{1-(\psi_i^{(h)})^2}$. The last
equality follows from IMS localization formula
\begin{equation*}
J_1A_1J_1 + J'_1A_1J'_1 -A_1=
h\sum_{i=1}^{N_0}(|d\psi_i|^2+|d\tilde{\psi_i}|^2)+h\sum_{j=1}^{N}(|d\phi_j|^2+|d\tilde{\phi_j}|^2)
\end{equation*}
(see (11.37) in \cite{CFKS}). All of $|d\psi_i^{(h)}|$,
$|d\tilde{\psi}_i^{(h)}|$, $|d\phi_j^{(h)}|$ and
$|d\tilde{\phi}_j^{(h)}|$ are $O(h^{-\kappa})$; therefore, the
inequality \eqref{e:15} is satisfied for $Q_1$ and
$\gamma_1=O(h^{1-2\kappa})$.
\\

Now we will define the operators $i_2:\mathcal{H}_ 0\rightarrow
\mathcal{H}_2$ and $p_{2}: \mathcal{H}_2 \rightarrow
\mathcal{H}_0$. Let $\phi_n ~ \in C^\infty(\mathbb{R}_{-})$ such
that $0 \leq \phi_n \leq 1$, $\phi_n(x) = 1$ if $-1 \leq x \leq
0$, $\phi_n(x) = 0$ if $x \leq -2$. For small enough $h$,
$\phi_n^{(h)} (x)= \phi(h^{-\kappa} x) \in C_c^\infty((-r, 0])$.
Let
\begin{equation}\label{E:alpha}
\alpha=C(h) \phi_n^{(h)}(x_n)e^{\frac{x_n}{h}} dx_n
\end{equation}
 be a 1-form on $\mathbb{R}_{-}$. We choose the constant  $C(h)$
 so that the form $\alpha $ has $L^2$-norm one  with respect
 to the metric $g_{nn}(0) dx_{n}^2$.\\

On $L^2(B_i(0, r), \Lambda T^*\mathbb{R}^n|_{B_i(0,
r)})$ we define  $i_2$ to be the inclusion given by the choice of
the special coordinates near $\bar x_i$ in $B_i(0, r)$, and on\\
$L^2(B_j(0, r), \Lambda T^*\mathbb{R}^{n-1}|_{B_j(0, r)})$ it is
defined by
\begin{equation}\label{E:i2}
i_2{\omega}= i(\alpha \wedge \omega)
\end{equation}
where $i: L^2(C_j(0, r), \Lambda T^*\mathbb{R}^{n}_{-}|_{C_j(0,
r)}) \rightarrow L^2(M, \Lambda T^*M)$ is the inclusion given by
the choice of the special coordinates near each $\bar y_j$ in
$C_j(0, r)$.\\

Let $L$ be the subspace of $L^2(M, \Lambda T^*M)$ which contains
only the forms which are of the form $\alpha \wedge \omega'(x')$
in the special coordinates around the boundary critical points
where $\omega'$ is a form which depends only tangential
components. The map $p_2$ is the composition of the restriction
map
\begin{align*}
p: &L^2(M, \Lambda T^*M)\rightarrow \\
&(\oplus_{i=1}^{N_0} L^2(B_i(0, r), \Lambda
T^*\mathbb{R}^n|_{B_i(0, r)})) \oplus (\oplus_{j=1}^{N} L^2(C_j(0,
r), \Lambda T^*\mathbb{R}^{n}_{-}|_{C_j(0, r)})),
\end{align*}
and the map $r: L^2(C_j(0, r), \Lambda
T^*\mathbb{R}^{n}_{-}|_{C_j(0, r)})\rightarrow L^2(B_j(0, r),
\Lambda T^*\mathbb{R}^{n-1}|_{B_j(0, r)})$ defined by the natural
extension of the map $r(\omega(x))= \frac{1}{C(h)} [i_{\partial /
\partial x_n}(\omega)](x',0)$ on the subspace of the smooth forms $\omega=
\alpha \wedge \omega'(x')$  to the closure of this subspace in
$L^2(C_j(0, r),\Lambda T^*\mathbb{R}^{n}_{-}|_{C_j(0, r)})$, and
zero on the orthogonal complement. In other words, $p_2$ is
defined by the formula
\begin{equation}\label{E:p2}
p_2(\omega(x))= \frac{1}{C(h)} [i_{\partial / \partial
x_n}p(\omega)](x',0)
\end{equation}
for any smooth $\omega \in L \subset L^2(M, \Lambda T^*M)$ such
that {\small$p(\omega) \in L^2(C_j(0, r), \Lambda
T^*\mathbb{R}^{n}_{-}|_{C_j(0, r)})$}.\\

It is easy to check that  $p_2\circ i_2={\rm id}_{\mathcal{H}_0}$
and the operator $J_2=i_2Jp_2$  maps the domain of the quadratic
form $Q_2$ to itself.\\

To show that the inequality \eqref{e:15} is satisfied for $Q_2$,
first
 by \eqref{E:MainQuad2} we have,
\begin{align*}
&Q_2(J_2 \omega) + Q_2(J_2' \omega)-Q_2(\omega) = (A_2J_2
\omega,J_2 \omega)_g+ (A_2J'_2 \omega,J'_2 \omega)_g-(A_2 \omega,
\omega)_g\\ &- \int_{\partial M}(\mathbf{t} d^*_{h,f,g} J_2\omega)
\wedge \star \mathbf{n} \overline{J_2\omega} - \int_{\partial
M}(\mathbf{t} d^*_{h,f,g} J'_2\omega) \wedge \star \mathbf{n}
\overline{J'_2\omega} + \int_{\partial M}(\mathbf{t} d^*_{h,f,g}
\omega) \wedge \star \mathbf{n} \bar\omega
\end{align*}
On $L$, we have
$$\mathbf{t}d^*_{h,f,g}(J_2\omega)=
d^*_{h,f,g}(\alpha)\wedge \phi_j^{(h)}\omega'$$ and
$$\mathbf{t}d^*_{h,f,g}(J'_2\omega)=
d^*_{h,f,g}(\alpha)\wedge \tilde{\phi}_j^{(h)}\omega'.$$
Therefore, the multiplication with $J_2$ and $J'_2$ commutes with
$d^*_{h,f,g}$ and we have
\begin{equation}\label{E:boundaryint}
-(\mathbf{t} d^*_{h,f,g} J_2\omega) \wedge (\star \mathbf{n}
\overline{J_2\omega}) - (\mathbf{t} d^*_{h,f,g} J'_2\omega) \wedge
(\star \mathbf{n} \overline{J'_2\omega}) + (\mathbf{t} d^*_{h,f,g}
\omega) \wedge (\star \mathbf{n} \bar\omega)=0
\end{equation}
on the boundary. On the orthogonal complement of $L$, around the
boundary\\ $J_2 \omega = 0 $ and $J'_2 \omega =\omega$; therefore
the equation \eqref{E:boundaryint} is still valid. So we have
\begin{align*}
&Q_2(J_2 \omega) + Q_2(J_2' \omega)-Q_2(\omega) = (A_2J_2
\omega,J_2 \omega)_g+ (A_2J'_2 \omega,J'_2 \omega)_g-(A_2 \omega,
\omega)_g\\
&= h\sum_{i=1}^{N_0}((|d\psi_i^{(h)}|^2 +
|d\tilde{\psi}_i^{(h)}|^2)\omega, \omega )_g+ \sum_{j=1}^{N}
((|d\phi_j^{(h)}|^2 + |d\tilde{\phi}_j^{(h)}|^2)\omega, \omega )_g
\end{align*}
by IMS localization formula (see(11.37) in \cite{CFKS}). Since all
$|d\psi_i^{(h)}|$, $|d\tilde{\psi}_i^{(h)}|$,
$|d\phi_j^{(h)}|$ and $|d\tilde{\phi}_j^{(h)}|$ are
$O(h^{-\kappa})$, the inequality \eqref{e:15} is satisfied for
$Q_2$ and $\gamma_2=O(h^{1-2\kappa})$.\\

We note that the operator $J_2$ is not self-adjoint, however it is
symmetric on $\textrm{Im}(J_2)$.\\

The following lemma provides the localization of eigenforms near
the points in $S_0 \cup S_+$. The proof of this lemma is similar
to the proof of Theorem 3.2.3 (p. 34 in \cite{HN}).\\

\begin{lemma}\label{L:alpha} Let $E$ be the complement of
the union of balls $B(\bar{x}_i, h^\kappa),~ i=1,...,N_0$, and $C(
\bar{y}_j, h^\kappa),~ j=1,...,N$. If $\omega \in D(Q_{h,f,g})$
such that $\supp(\omega)\subseteq E$, then there exists $h_0>0$
such that  for all $h \in (0, h_0)>0$
\begin{equation}\label{E:alpha}
Q_{h,f,g}(\omega) \geq C h^{2\kappa - 1}\|\omega\|_{g}^2
\end{equation}
for some constant $C>0$.
\end{lemma}
\bigskip

Since $\kappa < 1/2$, $h^{2 \kappa -1} \rightarrow \infty$ as $h
  \rightarrow 0$, the lemma implies that for the eigenforms with
 bounded eigenvalues it is enough to consider only the forms
 supported in a small neighborhood of the critical
 points in $S_0 \cup S_+$.\\

\begin{lemma}\label{L:modalpha} Let $\omega \in D(Q_1)$ such
that
$$
\supp(\omega)\subseteq \{x \in \mathbb{R}^{n}: \dist(x,0)
\geq h^\kappa\}^{N_0} \cup \{x \in \mathbb{R}^{n-1}: \dist(x, 0)
 \geq h^\kappa\}^N.
$$
 Then there exists $h_0>0$ such that for all $h
\in (0, h_0)$
 \begin{equation}\label{E:modalpha}
Q_1(\omega) \geq C h^{2\kappa - 1}\|\omega\|^2.
\end{equation}
\end{lemma}

% \smallskip
\begin{proof} We can write the quadratic form $Q_1$ as
 \begin{equation*}
Q_1(\omega) = (\Delta_{\textrm{mod}} \omega,\omega)_1
=\sum_{i=1}^{N_0}( \Delta_{i} \omega, \omega)_{g_i} +
\sum_{j=1}^{N}( \Delta_{j} \omega, \omega)_{g'_j}.
\end{equation*}
Suppose $\supp(\omega)$ is not empty in the domain of $\Delta_i$
for some $i\in \{1,...,N_0\}$.  Since all the operators in the
definition of the quadratic form $Q_1$ are positive operators, we
have
\begin{eqnarray*}
Q_1(\omega) &\geq& ( \Delta_{i} \omega, \omega)_{g_i}\\ &=&
h\left(\parallel d \omega
\parallel_{g_i} ^2 +\parallel d^{*} \omega
\parallel_{g_i} ^2\right)+\left((\mathcal{L}_{\nabla f} +
\mathcal{L}_{\nabla f}^{*}) \omega , \omega \right)_{g_i}
+h^{-1}\left(V_i \omega ,\omega \right)_{g_i} \\&\geq&
h^{-1}\left(V_i \omega ,\omega \right)_{g_i}.
\end{eqnarray*}
We also have that $V_i \geq C_1 h^{2 \kappa}$ for some $C_1>0$.
Therefore, we conclude that  $Q_{1} (\omega) \geq C h^{2\kappa-1}
\parallel \omega \parallel^2$.\\

If $\supp(\omega)$ is empty in the domain of $\Delta_i$ for all
$i$, then it is not empty in the domain of $\Delta_j$ for some
$j\in\{1,...,N\}$. Since $V' \geq C_2 h^{2 \kappa}$ for some
$C_2>0$, similar argument will lead us the inequality
\eqref{E:modalpha}.
\end{proof}
\bigskip

Our next goal is to obtain the estimates (\ref{e:16}) and
(\ref{e:A1A2}). In the setting of this section these estimates are
equivalent to the inequalities
\begin{equation}\label{E:betaepsilon01}
Q_2 (i_2 \phi_j^{(h)}(x)  \omega) \geq (1-C_1 h^{\kappa}) Q_1 (i_1
\phi_j^{(h)}(x)  \omega) -C_2 h^{3 \kappa -1}
\parallel i_1\phi_j^{(h)}(x)  \omega
\parallel_{g'_j } ^2
\end{equation}
and
\begin{equation}\label{E:betaepsilon02}
Q_1 ( i_1 \phi_j^{(h)}(x)  \omega) \geq (1-C_1 h^{\kappa}) Q_2 (
i_2 \phi_j^{(h)}(x)  \omega) -C_2 h^{3 \kappa- 1}
\parallel  i_2 \phi_j^{(h)}(x)  \omega
\parallel_{g} ^2,
\end{equation}
where  $\omega \in D \cap \left(\oplus_{j=1}^{N} L^2(B_j(0, r),
\Lambda T^*\mathbb{R}^{n-1}|_{B_j(0, r)})\right)$ and $C_1, C_2
>0$ are some constants. Remember that $$D=\{\omega\in {\mathcal
H}_0 : i_1J\omega\in D(Q_1)\}=\{\omega\in {\mathcal H}_0 :
i_2J\omega\in D(Q_2)\}.$$ In the proof of the inequalities we will
use some intermediate quadratic forms
in order to compare $Q_1$ and $Q_2$.\\

At each critical point $\bar y_j \in S_+$, let $x_1,...,x_n$ be
the special coordinates. In these coordinates $f$ and $g$ can be
written as \eqref{E:f} and \eqref{E:g} respectively. We use
the formula \eqref{E:f'}  to  extend $f$ to $\mathbb{R}_{-}^{n}$.
Let
$\tilde{g}_j$ be an extension of the metric $g$ to
$\mathbb{R}_{-}^{n}$ such that $\tilde{g}_j(x) = g(x)$ if
$|x'|\leq h^\kappa$ and $|x_n| \leq C$ for some positive constant
$C$, and
$\tilde{g}_j(x)=g(0)$ if $|x'|\geq 2h^\kappa$ or $|x_n| \geq 2C$.\\

Note that in the set $|x'| \leq h^{\kappa}$, $|x_n| \leq
h^{\kappa}$, $f=f_j$ and $g=\tilde{g}_j$, therefore for each $\bar
y_j \in S_+$ and for any smooth $\omega \in \mathcal{H}_0$ we have
\begin{equation}\label{E:equality}
Q_2(i_2(\phi_j^{(h)}(x) \omega)) = Q_{h,f_j,\tilde{g}_j}(
\alpha\wedge i_1(\phi_j^{(h)}(x) \omega)),
\end{equation}
where $Q_{h,f_j,\tilde{g}_j}$ is a quadratic form
on $H^1(\mathbb{R}^{n}_{-}; \Lambda T^*\mathbb{R}^{n}_{-})$.\\

The proof of the following lemma is similar to the proof of
Lemma 3.3.7 in \cite{HN}.\\

\begin{lemma} Let $f= f_j$, $g_1 = \tilde{g}_j$, and $g_2(x', x_n) = \tilde{g}_j (x',0)$.
Then for some constant $C \geq 0$,
 \begin{equation}\label{E:app1}
 Q_{h,f, g_1} (\omega) \geq (1-C h^{\kappa})
 Q_{h,f,g_2} (\omega) -C h^{\kappa}
 \parallel \omega \parallel_{g_2} ^2
 \end{equation}
for $\omega \in H^1(\mathbb{R}^{n}_{-}; \Lambda
T^*\mathbb{R}^{n}_{-})$ such that $\mathbf{t}(\omega)=0$ and
$\supp\,\omega \subset \{x_n \geq -C_0 h^\kappa \}$ for some
constant $C_0>0$.
\end{lemma}
\bigskip

Note that we also have
\begin{equation}\label{E:app2}
Q_{h,f_j,g_2} (\omega) \geq (1-C h^{\kappa}) Q_{h,f_j,g_1}
(\omega) -C h^{\kappa}
 \parallel \omega \parallel_{g_1} ^2
\end{equation}
since the argument in the proof of Lemma 3.3.7 is symmetric.\\
\bigskip

Now we want to freeze $(g_2)_{nn}$. To do this we will need the following lemma.
\begin{lemma}\label{L:Const} Let $\displaystyle g_3(x)=
\frac {(g_2)_{nn}(0)}{(g_2)_{nn}(x')} ~g_2(x)$.
 Then there exist $C_1, C_2>0$ and $h_0$ such that
for any $h \in (0, h_0)$
\begin{equation}\label{E:app3}
Q_{h,f,g_3} (\omega) \geq (1-C_1 h^{\kappa}) Q_{h,f,g_2} (\omega)
-C_2 h^{1 - \kappa} \parallel \omega \parallel_{g_2} ^2
\end{equation}
for $\omega$ such that ${\rm supp}\, \omega \subset \{ x \in
\mathbb{R}^{n}_{-}: |x'| \leq C_0 h^{\kappa},~ x_n \geq - C_0
h^{\kappa}\}$ for some constant $C_0>0$.
\end{lemma}

\begin{proof} Let $\displaystyle e^{\varphi(x)} = \frac
{(g_2)_{nn}(0)}{(g_2)_{nn}(x')}$, then $g_3(x) = e^{\varphi(x)}
g_2(x)$. Note that $\varphi(0) = 0$. Since $g_2 =g_3$ when $|x'|
\geq 2 h^\kappa$, and both $g_2$ and $g_3$ are $x_n-$independent,
$\varphi(x) = O(h^{\kappa})$ and $e^{\varphi(x)} = 1 +
O(h^\kappa)$ everywhere. Therefore,
\begin{equation}\label{E:estvarphi}
\min\{e^{\varphi(x)}\} = 1 + O(h^{\kappa}).
\end{equation}
Given a metric $g$, the volume form is $V_{g}(x)= (\det
g(x))^{\frac{1}{2}} dx_1\wedge ... \wedge dx_n$. Then,
\begin{equation*}
V_{g_3}(x)= e^{\frac{n}{2} \varphi(x)} V_{g_2}(x).
\end{equation*}
For any $p$-forms $\omega$ and $\eta$
\begin{equation*}
%\label{E:}
\omega \wedge *_{g_3}\eta = e^{(\frac{n}{2}-p)\varphi(x)}\omega
\wedge *_{g_2}\eta,
\end{equation*}
see (pp. 26-27 \cite{HN}). Therefore,
\begin{equation*}
%\label{E:}
\parallel \omega \parallel _{g_3} \geq e^{(\frac{n}{2}-p)\min\varphi(x)}
\parallel \omega \parallel _{g_2}.
\end{equation*}
\smallskip

Since $d_{h,f}$ does not depend on $g$, we have
\begin{equation*}
%\label{E:}
\parallel d_{h,f} \omega \parallel_{g_3} ^2 \geq
e^{(\frac{n}{2}-p-1)\min\varphi(x)} \parallel d_{h,f} \omega
\parallel_{g_2} ^2.
\end{equation*}
Thus by \eqref{E:estvarphi},
\begin{equation}\label{E:estd}
\parallel d_{h,f} \omega \parallel_{g_3} ^2 \geq (1-C'_1
h^{\kappa}) \parallel d_{h,f} \omega
\parallel_{g_2} ^2.
\end{equation}
\smallskip

We also have
\begin{equation}\label{E:dstar}
\parallel d^{*}_{h,f,g_3} \omega \parallel_{g_3} ^2 \geq
e^{(3p-1+n)\min\varphi(x)} \parallel d^{*}_{h,f,g_2} \omega +h
i_{\nabla (\frac{n}{2}-p)\varphi(x)}\omega \parallel_{g_2} ^2,
\end{equation}
see the proof of Lemma 3.3.8 in \cite{HN}.\\

 Now we will use the following well known inequality,
\begin{equation*}
%\label{E:}
\parallel f+g \parallel ^2 \leq
(1+\varepsilon) \parallel f \parallel ^2+ (1+
\frac{1}{\varepsilon})\parallel g
\parallel ^2,
\end{equation*}
which implies that
\begin{equation*}
%\label{E:}
\parallel f \parallel ^2 \geq
\frac{1}{1+\varepsilon} \parallel f+g \parallel ^2- \frac{1}
{\varepsilon}\parallel g
\parallel ^2.
\end{equation*}
Replacing $\varepsilon$ by $h^ \kappa$, $f$ by $f+g$ and $g$
by $-g$, we get
\begin{equation*}
%\label{E:}
\parallel f+g \parallel ^2 \geq
(1-C_1 h^\kappa) \parallel f \parallel ^2- C_2
h^{-\kappa}\parallel g \parallel ^2,
\end{equation*}
which implies that
\[\begin{array}{rl}
\parallel d^{*}_{h,f,g_2} \omega &+h i_{\nabla
(\frac{n}{2}-p)\varphi(x)}\omega  \parallel_{g_2} ^2\\&
\\&\geq (1-C_1 h^\kappa) \parallel d^{*}_{h,f,g_2} \omega
\parallel_{g_2} ^2- C_2 h^{-\kappa}\parallel h i_{\nabla
(\frac{n}{2}-p)\varphi(x)}\omega \parallel_{g_2}^2\\&\\
 &\geq (1-C_1 h^\kappa) \parallel d^{*}_{h,f,g_2} \omega
 \parallel_{g_2} ^2- C_2 h^{2-\kappa}\parallel \omega \parallel_{g_2}
 ^2.
\end{array}
\]
This together with the equation \eqref{E:estvarphi} and the
inequality \eqref{E:dstar} imply

\begin{equation*}
\parallel d^{*}_{h,f,g_3} \omega \parallel_{g_3} ^2 \geq (1-C_1
h^{\kappa})\parallel d^{*}_{h,f,g_2} \omega
\parallel_{g_2} ^2 -C_2 h^{2 - \kappa} \parallel \omega
\parallel_{g_2} ^2.
\end{equation*}
Together with \eqref{E:estd} we have
\begin{eqnarray*}
Q_{h,f,g_3} (\omega) &=& \frac{1}{h}(\parallel d_{h,f} \omega
\parallel _{g_3} ^2 +\parallel d^{*}_{h,f,g_3} \omega
\parallel _{g_3}^2)\\
&\geq& (1-C_1 h^{\kappa}) Q_{h,f,g_2} (\omega) -C_2 h^{1 - \kappa}
\parallel \omega \parallel_{g_2} ^2.
\end{eqnarray*} \end{proof}
\bigskip

Note that for the inequality
\begin{equation}\label{E:app4}
Q_{h,f,g_2} (\omega) \geq (1-C_1 h^{\kappa}) Q_{h,f,g_3} (\omega)
-C_2 h^{1 - \kappa} \parallel \omega \parallel_{g_3} ^2
\end{equation}
it is enough to see that
\[\min\{e^{-\varphi(x)}\} =
(\max\{e^{\varphi(x)}\})^{-1}= (1 + O(h^{\kappa}))^{-1} = 1 +
O(h^{\kappa}).\]\\

Now we will continue with the proof of the main result.\\

We combine formula \eqref{E:equality} with inequalities \eqref{E:app1}, \eqref{E:app2},
\eqref{E:app3}, and \eqref{E:app4}, to conclude that there exist positive constants $C_1$ and $C_2$
such that

\begin{align}\label{E:app5}
Q_2 (i_2(\phi_j^{(h)}(x) \omega)) \geq& (1-C_1 h^{\kappa})
Q_{h,f_j,g_3} (i_1(\phi_j^{(h)}(x) \omega)\wedge \alpha)\\ &-C_2
h^{\kappa}\parallel i_1(\phi_j^{(h)}(x) \omega)\wedge
\alpha\parallel_{g_3} ^2 \notag
\end{align}
and
\begin{align}\label{E:app6}
Q_{h,f_j,g_3} (i_1(\phi_j^{(h)}(x) \omega)\wedge \alpha) \geq&
(1-C_1 h^{\kappa}) Q_2 (i_2(\phi_j^{(h)}(x) \omega))\\ &-C_2
h^{\kappa} \parallel i_2(\phi_j^{(h)}(x) \omega)\parallel_{g} ^2
\notag
\end{align}
for all $\omega \in \mathcal{H}_0$.\\

Note that $(g_3)_{nn} = g_{nn}(0)$ and $(g_3)_{lk}$ depends only
on $x'$ for $l,k = 1, ... , n-1$. Therefore, in the metric $g_3$ the
variables separate on the forms $\psi =i_1(\phi_j^{(h)}(x) \omega)\wedge
\alpha $.  In particular, let $ f'_j =
f\mid_{\partial M}$, $g'_3$ be  the restriction of the metric $g_3$
on the tangent space spanned by $\{{\partial /
\partial x_1}, {\partial / \partial x_2},...,{\partial / \partial
x_{n-1}}\}$, $f_j^{(n)} = x_n$, $g_{3}^{(n)} = g_{nn}(0) dx_n^2$.
Then for any $\omega $ from the component of $\mathcal{H}_0$
corresponding to the boundary critical points we have
\begin{align}\label{E:sep}
Q_{h, f_j, g_3}((i _1 \phi_j^{(h)}(x') \omega) \wedge \alpha) = &
\,Q_{h, f'_j, g'_3}(i _1 \phi_j^{(h)}(x')\omega) \parallel \alpha
\parallel_{g_3^{(n)}} ^2\\
&+\parallel i _1 \phi_j^{(h)}(x') \omega \parallel_{g'_3} ^2 Q_{h,
f_j^{(n)}, g_3^{(n)}}(\alpha)\notag.
\end{align}

We recall that from normalization we have $\parallel \alpha
\parallel_{g_3} ^2 =1$. Moreover,\\
$Q_{h, f_j^{(n)}, g_{3}^{(n)}}
(\alpha)\leq C h^{3 \kappa - 1}$ for some $C>0$.
Indeed,
\begin{equation*}
Q_{h, f_j^{(n)}, g_{3}^{(n)}}(\alpha)= \frac{1}{h} \parallel
d^{*}_{h, f_j^{(n)}, g_{3}^{(n)}} \alpha  \parallel
_{g_{3}^{n}}^{2},
\end{equation*}
where the norm and the adjoint are taken with respect to the metric
$g_{3}^{(n)}$ on $\mathbb R_{-}^{n}$. Since $d^{*}_{h, f_j^{(n)},
g_{3}^{(n)}}(e^{\frac{x_n}{h}}dx_n)=0$, we have
\begin{equation*}
Q_{h, f_j^{(n)}, g_{3}^{(n)}}(\alpha)\leq C_1 h(c(h))^2\parallel e^{\frac{x_n}{h}}\frac{\partial
\phi^{(h)}_n}{\partial x_n}\parallel ^{2}_{g_{3}^{(n)}},
\end{equation*}
where $C_{1}$ is a constant independent of $h$, and $c(h)$ is normalization constant for
$\phi^{(h)}_n(x_n)e^{\frac{x_n}{h}}dx_n$ with respect to the
metric $g_{3}^{(n)}$.
 Since the support of $\displaystyle\frac{\partial
\phi^{(h)}_n}{\partial x_n}$ is in the interval $(-2h^{\kappa},
-h^{\kappa})$,
\begin{equation*}
Q_{h, f_j^{(n)}, g_{3}^{(n)}}(\alpha) \leq P(h) e^{-h^{\kappa -1}}
\end{equation*}
where $P(h)$ is a polynomial. Thus there exist $h_0>0$ small
enough so that for any $h \in (0,h_0)$, $Q_{h, f_j^{(n)},
g_{3}^{(n)}} (\alpha)\leq C h^{3 \kappa - 1}$ for some $C>0$.
 Therefore,
\begin{align}\label{E:app7}
Q_{h, f_j, g_3}((i _1 \phi_j^{(h)}(x) \omega) \wedge \alpha) \leq&
\,Q_{h, f'_j, g'_3}(i _1 \phi_j^{(h)}(x) \omega)\\
&+ C h^{3\kappa-1}
\parallel i _1 \phi_j^{(h)}(x) \omega
\parallel_{g'_3} ^2,\notag
\end{align}
and
\begin{align}\label{E:app8}
Q_{h, f'_j, g'_3}(i _1 \phi_j^{(h)}(x) \omega) \leq&\, Q_{h, f_j,
g_3}((i _1 \phi_j^{(h)}(x) \omega) \wedge \alpha)\\
&+ C h^{3\kappa-1}
\parallel i _1 \phi_j^{(h)}(x) \omega \wedge \alpha
\parallel_{g_3} ^2\notag.
\end{align}\\

Now we will compare the quadratic forms $Q_{h, f'_j, g'_3}$ and
$Q_j$. The quadratic forms $Q_{h, f'_j, g'_3}$ can be written as
\begin{equation*}
%\label{E:}
\begin{aligned}
Q_{h,f'_j,g'_3}(\omega)= &h \sum_{l,k=1}^{n-1} \left(
A^{(2)}_{2,lk} \frac{\partial}{\partial x_{l}} \omega,
\frac{\partial}{\partial x_{k}}\omega \right) _{g'_3} + h
\sum_{l=1}^{n-1} \left(A^{(1)}_{2,l} \frac{\partial}{\partial
x_{l}} \omega, \omega\right) _{g'_3}\\
&\displaystyle +(B_{2} \omega , \omega)_{g'_3}+ h^{-1} (V_{2}
\omega, \omega)_{g'_3}
\end{aligned}
\end{equation*}
where $A^{(2)}_{2,lk} = (g'_3)^{lk}(x')$, $A^{(1)}_{2,l}$ is a
first order operator. Operator $B_2$ is bounded, and it can be written as
$$B_{2} = \sum_{l,k=1}^{n-1} \frac{\partial^2
f'_j}{\partial x_{l}\partial x_{k}}(x')~ (a^{*k} a^{l}-a^l a^{*k})
- \sum_{l}^{n-1}(\nabla f'_j)_l(x') B'_{l}(x').$$ Finally,
$$V_{2}= \sum_{l,k=1}^{n-1} (g'_3)^{kl}(x') \frac{\partial
f'_j}{\partial x_{l}}(x')\frac{\partial
f'_j}{\partial x_{k}}(x').$$\\

In coordinates the quadratic form $Q_1$ for the model
operator $\Delta_{mod}$ (see \eqref{E:ModelOp}) on
$L^2(\mathbb{R}^{n-1}, \Lambda T^*\mathbb{R}^{n-1})$ corresponding
to $\bar{y}_j$ can be written as
\begin{equation*}
%\label{E:}
Q_1(\omega)= h(A^{(2)}_{1,lk} \frac{\partial}{\partial x_{l}}
\omega, \frac{\partial}{\partial x_{k}} \omega)_{g_j'} + (B_1
\omega , \omega)_{g_j'}+ h^{-1} (V_1 \omega, \omega)_{g_j'}
\end{equation*}
where $A^{(2)}_{1,lk} = (g_j')^{lk}$,
$$ B_1 =\sum_{l,k=1}^{n-1}
 \frac{\partial^2 f_j'}{\partial x_{l}\partial
x_{k}}(\bar y_j)~ (a^{*k} a^{l}-a^l a^{*k}),$$ and
$$V_1=\sum_{l,k,r,s=1}^{n-1} (g_j')^{rs} \frac{\partial^2
f_j'}{\partial x_{r} \partial x_{l}}(\bar{y}_j) \frac{\partial^2
f_j'}{\partial x_{s} \partial x_{k}}(\bar{y}_j) x_l x_k.$$\\

Since $g'_3(0)=g_j'$, $f'_j$ is a Morse functions on
$\partial M$, and $\bar{y}_j$ corresponds to the origin, we have that
\begin{eqnarray*}
A_{2,lk}(0)=A_{1,lk}(0),\quad B_2(0)=B_1(0),\quad V_2(0)=V_1(0)=0,\\
\frac{\partial V_2}{\partial x^l}(0)=\frac{\partial V_1}{\partial
x^l}(0)=0, \quad l=1,2,\ldots,n,\\ \frac{\partial^2 V_2}{\partial
x^l\partial x^k}(0)=\frac{\partial^2 V_1}{\partial x^l \partial
x^k}(0), \quad l,k=1,2,\ldots,n.
\end{eqnarray*}\\
These equalities together with Lemma 2.11 \cite{KMS} imply the
following lemma. (Lemma 2.11 is proved in \cite{KMS} in the case
of operators,
but its proof can be easily extended to include the case of quadratic forms).\\

Let $\phi \in C_c^\infty (\mathbb{R}^{n-1})$  be a function
satisfying $0 \leq \phi \leq 1$, $\phi(x)=1$ if $|x| \leq 1$,
$\phi(x)=0$ if $|x| \geq 2$.
%Let $\phi' = (1 - \phi^2)^{1/2} \in
%C^\infty (\mathbb{R}^{n-1})$, and
Define $\phi^{(h)}(x)= \phi(h^{-\kappa}x)$.\\

\begin{lemma}\label{L: n-1 part} Let $1/3 < \kappa < 1/2$.
There exist $C>0$ and $h_0$ such that
for any $h \in (0, h_0)$,
\begin{equation}\label{E: n-1 part}
Q_{h, f'_j, g'_3} (\phi^{(h)}\omega ) \leq (1+C h^{\kappa}) Q_{j}
(\phi^{(h)} \omega ) +C h^{3 \kappa -1} (\phi^{(h)} \omega
,\phi^{(h)} \omega )_{g'_j}
\end{equation}
for $\omega \in C_c^\infty (\mathbb{R}^{n-1}; \Lambda
T^*\mathbb{R}^{n-1})$ such that ${\rm supp}\, \omega \subset \{ x
\in
\mathbb{R}^{n-1}: |x| \leq C_0 h^{\kappa}\}$ for some constant $C_0>0$.\\
\end{lemma}

The proof of Lemma 2.11 \cite{KMS} is symmetric with respect to
the quadratic forms $Q_{h, f'_j, g'_3}$ and $Q_{j}$, so after dividing by $(1+Ch^{\kappa })$, we have that
\begin{equation*}
Q_{h, f'_j, g'_3} (i_1 \phi_j^{(h)}(x) \omega) \geq (1-C_1
h^{\kappa}) Q_{j} (i_1 \phi_j^{(h)}(x) \omega) -C_2 h^{3\kappa -1}
\parallel i_1 \phi_j^{(h)}(x) \omega \parallel_{g'  _j} ^2
\end{equation*}
and
\begin{equation*}
Q_{j} (i_1 \phi_j^{(h)}(x) \omega) \geq (1-C_1 h^{\kappa}) Q_{h,
f'_j, g'_3} (i_1 \phi_j^{(h)}(x)\omega) -C_2 h^{3\kappa -1}
\parallel i_1 \phi_j^{(h)}(x)\omega \parallel_{g'_3} ^2
\end{equation*}
for any $\omega \in D$ and for some $C_1, C_2 >0$.\\

All the metrics we used differ from each other by multiplication
by $(1+O(h^{\kappa}))$ in $h^{\kappa}$-neighborhood of the points
in $S$. Therefore these two inequalities together with
\eqref{E:app5}, \eqref{E:app6}, \eqref{E:app7}, \eqref{E:app8}
imply \eqref{E:betaepsilon01} and \eqref{E:betaepsilon02}.\\

The comparison of quadratic forms around the interior critical
points is completely similar to the comparison of the  forms
$Q_{h, f'_j, g'_2}$ and $Q_j$. Thus for any interior point $x_i
\in S_0$, Lemma 2.11 \cite{KMS} implies that
\begin{equation}\label{E:intbetaepsilon01} Q_{h,f,g} (i_2
\psi_i^{(h)}(x)  \omega) \geq (1-C_1 h^{\kappa}) Q_{i} (i_1
\psi_i^{(h)}(x)  \omega) -C_2 h^{3 \kappa -1}
\parallel i_1\psi_i^{(h)}(x)  \omega
\parallel_{g_i} ^2
\end{equation}
and
\begin{equation}\label{E:intbetaepsilon02}
Q_{i} ( i_1 \psi_i^{(h)}(x)  \omega) \geq (1-C_1 h^{\kappa})
Q_{h,f,g} ( i_2 \psi_i^{(h)}(x)  \omega) -C_2 h^{3 \kappa- 1}
\parallel  i_2 \psi_i^{(h)}(x)  \omega
\parallel_{g} ^2
\end{equation}
for $\omega \in D \cap (\oplus_{i=1}^{N_0} L^2(B_i(0, r), \Lambda
T^*\mathbb{R}^n|_{B_i(0, r)}))$ and for some $C_1, C_2 >0$.
Indeed, for interior critical point
these estimates are the same as in the case of manifolds without boundary  (see \cite{S1}). \\

Therefore we have
\begin{equation}\label{E:betaepsilon1}
Q_2 (i_2 J \omega) \geq (1-C_1 h^{\kappa}) Q_1 (i_1 J \omega) -C_2
h^{3 \kappa -1}\parallel i_1 J \omega \parallel_{g_j } ^2
\end{equation}
and
\begin{equation}\label{E:betaepsilon2}
Q_1 ( i_1 J \omega) \geq (1-C_1 h^{\kappa}) Q_2 ( i_2 J  \omega)
-C_2 h^{3 \kappa- 1} \parallel  i_2 J \omega \parallel_{g} ^2
\end{equation}
for $\omega \in D$ and for some $C_1, C_2 >0$. These inequalities imply \eqref{e:16} and \eqref{e:A1A2}.\\

The following lemma verifies inequality \eqref{e:14}.
\begin{lemma}\label{L:realalpha} Let $\omega \in D(Q_2)$
 then there exists $h_0$ such that  for all $h \in
(0, h_0)$
\begin{equation*}
Q_2(J'_2\omega)\geq C h^{2\kappa -1} \|J'_2\omega\|_g^2
\end{equation*}
for some $C>0$.
\end{lemma}
\begin{proof} Let $\omega$ be a form supported on $C(\bar{y}_j, h^\kappa)$.
First we assume that $\omega$ restricts to a tangential form on the boundary, that is $\omega =
\mathbf{t}\omega$. Then the boundary integral term in
\eqref{E:MainQuad2} vanishes because of the boundary condition $\mathbf{t}\omega=0$.
Therefore,
\begin{equation*}
Q_{h,f,g}(\omega)=h\left(\parallel d \omega \parallel_{g} ^2
+\parallel d^{*} \omega
\parallel_{g} ^2\right)+\left((\mathcal{L}_{\nabla f} +
\mathcal{L}_{\nabla f}^{*}) \omega , \omega
\right)_{g}+h^{-1}\left(|\nabla f|^2
\omega ,\omega \right)_{g}.\\
\end{equation*}
Since $|\nabla f|> C$ around $\bar{y}_j$ for some positive
constant $C$, we have
$$h^{-1}\left(|\nabla f|^2 \omega ,\omega
\right)_{g} > C h^{-1}\|\omega\|_g^2.$$ Therefore,
for tangential forms we have that
\begin{equation}\label{e:realalphaapp1}
Q_2(J'_2\omega)\geq C h^{2\kappa -1} \|J'_2\omega\|_g^2
\end{equation}
for some $C>0$.\\

Now, let $\omega$ be a form that restricts to a normal form on the
boundary, that is $\omega = \mathbf{n}\omega$. In the special
local coordinates on $C(\bar{y}_j, h^\kappa)$, consider the forms
that can be written as  $\omega'(x') \wedge \tilde{\alpha}(x_n)$,
where $\tilde{\alpha}$ is a 1-form that belongs to the
$L^2$-orthogonal complement (with respect to the metric
$g_{nn}(0)dx_n^2$) of the one dimensional space generated by
$\alpha$ in the space of all 1-forms supported in the interval
$(-2h^{\kappa}, 0]$ in $H^1({\mathbb R}_{-}, \Lambda T^*{\mathbb
R}_{-})$. Since the inequalities \eqref{E:app1} and \eqref{E:app4}
are valid for any differential form with the support in a small
neighborhood of critical points on the boundary, we have that
\begin{equation*}
Q_2(\omega)= Q_{h,f_j, g_1} (\omega) \geq (1-C h^{\kappa})
Q_{h,f_j,g_3} (\omega) -C h^{\kappa}
\parallel \omega \parallel_{g_3} ^2.
\end{equation*}
Therefore,
\begin{equation*}
Q_2 (\omega' \wedge \tilde{\alpha})  \geq  (1-C_1 h^{\kappa})~
Q_{h,f_j,g_3} (\omega' \wedge \tilde{\alpha}) -C_2 h^{\kappa}
\parallel \omega' \wedge \tilde{\alpha} \parallel_{g_3} ^2.
\end{equation*}
After separating variables as in \eqref{E:sep}  and observing that
 $Q_{h,f'_j,g'_3}$ is a positive quadratic form, we obtain
\begin{equation*}
Q_2 (\omega' \wedge \tilde{\alpha}) \geq (1-C_1 h^{\kappa})\parallel \omega'
\parallel_{g'_3} ^2 Q_{h,f^{(n)}_j,g_3^{(n)}} (\tilde{\alpha}) -C_2 h^{\kappa}
\parallel \omega' \wedge \tilde{\alpha} \parallel_{g_3} ^2.
\end{equation*}
A simple calculation shows that the spectrum of the quadratic form
$Q_{h,f^{(n)}_j,g_3^{(n)}}$ on $H^1({\mathbb R}_{-}, \Lambda
T^*{\mathbb R}_{-})$ is $\{0\} \cup [Ch^{-1}, \infty)$. The
eigenspace corresponding to the $0$ eigenvalue is the one
dimensional space generated by the eigenform $\exp(x_{n}/h)dx_n$.
Since the forms $\alpha =C(h) \phi_n^{(h)}(x_n)\exp(x_{n}/h) dx_n$
(see \eqref{E:alpha} ) and $C(h)\exp(x_{n}/h) dx_n$ are equal for
$-h^{-\kappa}<x_n\leq 0$ and the function $\exp(x_{n}/h)$
decreases exponentially fast when $h\to 0$ and $x_n\leq
-h^{-\kappa}$,  we conclude that
\begin{equation*}
Q_{h,f^{(n)}_j,g_3^{(n)}} (\tilde{\alpha}) \geq h^{-1}
\parallel \tilde{\alpha} \parallel_{g_3^{(n)}} ^2
\end{equation*}
which implies that
\begin{equation*}
Q_2 (\omega' \wedge \tilde{\alpha}) \geq C h^{-1}
\parallel \omega' \wedge \tilde{\alpha} \parallel_{g_3} ^2.
\end{equation*}
\smallskip

Since any normal form supported in $C(\bar{y}_j, h^\kappa)$ which
belongs to the image of $J'_2$ can be approximated by the forms
$\omega'(x') \wedge \tilde{\alpha}(x_n)$ in special local
coordinates and the metrics $g$ and $g_3$ differ from each other
by $O(h^{\kappa})$, for any tangential form $\omega$ we have
\begin{equation}\label{e:realalphaapp2}
Q_2 (J'_2\omega ) \geq C h^{-1}
\parallel J'_2\omega \parallel_{g} ^2
\end{equation}
for some $C>0$.\\

The inequalities \eqref{e:realalphaapp1} and
\eqref{e:realalphaapp2} together with the lemma \ref{L:alpha}
imply the desired inequality.
\end{proof}
\smallskip

Now we can apply Theorem \ref{t:equivalence}. From
\eqref{E:betaepsilon1} and \eqref{E:betaepsilon2}, $\beta_l= 1+
O(h^{\kappa})$ and $\epsilon_l=O(h^{3\kappa -1})$ for $l= 1, 2$.
By the Lemma \ref{L:realalpha}, $\alpha_2 = O(h^{2 \kappa -1})$
and by the Lemma \ref{L:modalpha}, $\alpha_1 =O(h^{2 \kappa -1})$.
The operators $A_1= \Delta_{\textrm{mod}}$ and $A_2=\Delta_{h,f,g}$ are elliptic operators with positive definite principal symbols, thus $A_1$ and $A_2$
are bounded from below so $\lambda_{0l}=O(1)$.  Moreover, we have
$\gamma_l = O(h^{1- 2\kappa})$. Now assume that
$(a_1, b_1)$ does not intersect with the spectrum of $A_1$. For
$a_2$, $b_2$ given by formulas \eqref{e:a2}, \eqref{e:b2}
respectively, we have
\begin{equation*}
a_2 = a_1 + O(h^{s}), \quad b_2= b_1 + O(h^{s})
\end{equation*}
where $s= \min \{3\kappa -1, 1- 2\kappa\}$. The best possible
value of $s$ is
\begin{equation}\label{E:s}
s= \max_{\kappa} \min\{3\kappa -1, 1- 2\kappa\} = \frac{1}{5}
\end{equation}
which is attained when $\kappa = 2/5$. By Theorem
\ref{t:equivalence}, the interval $(a_2, b_2)$ does not intersect
with the spectrum of $A_2$. Moreover, for any $\lambda_1 \in (a_1,
b_1)$ and $\lambda_2 \in (a_2, b_2)$, $\dim(\textrm{Im}
E_1(\lambda_1)) = \dim(\textrm{Im} E_2(\lambda_2))$. Assume that
there are $M$ eigenvalues of the model operator $A_1$ lower than
$R$ and let $a_1$ be the highest eigenvalue of $A_1$ lower than
$R$. Since $R \notin \spec(A_1)$, there exists $h_0 > 0$ such that
for all $h \in (0, h_0)$, $a_2 = a_1 + O(h^{s})$ is less then $R$.
Then $\dim(\textrm{Im} E_1(R)) = \dim(\textrm{Im} E_2(R))$ which
implies $A_2$ also has exactly $M$ eigenvalues lower than $R$.
Since we can do this for any $R \notin \spec(A_1)$, we can
conclude that the eigenvalues of $A_2$ concentrates in the
$h^s$-neighborhood of the eigenvalues of $A_1$ and for any
$\lambda \in \spec(A_1)$, $A_2$ has exactly as many eigenvalues in
the $h^{s}$-neighborhood of $\lambda$ as the multiplicity of
$\lambda$. This implies Theorem \ref{T:SpecClose}.\\

\begin{remark} There exists an isomorphism between the spectral
spaces\\
${\rm Im} E_1(\lambda_1)$ and ${\rm Im} E_2(\lambda_1)$ which is
given in the proof of Theorem 3.1 in the Appendix.
\end{remark}

\section{Eigenvalues of the Model Operator}\label{S:EV}

Recall that the model operator does not depend on the choice of
local coordinates, see Section \ref{S:Pre}. Thus we will choose
local coordinates in which the model operator has an especially
simple form.\\

At each critical point $\bar{x}_i \in S_0$, let us choose Morse
coordinates $x_1, ... , x_n$ for $f$ near $\bar{x}_i$. In these
coordinates $\bar{x}_i =0$, the metric at $\bar{x}_i$ is Euclidean
i.e. $g_i= \sum_{r=1}^{n} dx_r^2$ and for some non vanishing real
constants $c_r$, $r= 1, ..., n-1$,
\[
f(x)= f(0) + \frac{1}{2} \sum_{r=1}^{n} c_r x_r^2.
\]
Let $f_i$ be the extension of $f$ to $\mathbb{R}^{n}$ with the
same
formula.\\

At each boundary critical point $\bar{y}_j \in S_+$, let us choose
Morse coordinates $x_1, ... , x_{n-1}$ for $f\mid_{\partial M}$ near
$\bar{y}_j$. In these coordinates $\bar{y}_j =0$, the metric at
$\bar{y}_j$ restricted to the tangential vectors is Euclidean i.e.
$g_j= \sum_{r=1}^{n-1} dx_r^2$ and for some non vanishing real
constants $d_r$, $r= 1, ..., n-1$,
\[
f\mid_{\partial M}(x)= f(0) + \frac{1}{2} \sum_{r=1}^{n-1} d_r
x_r^2.
\]
Let $f_j$ be the extension of $f\mid_{\partial M}$ to
$\mathbb{R}^{n-1}$ with the same formula.\\

In these coordinates the operators $\Delta_i$ and $\Delta_j$ can
be written as
\begin{equation*}
\Delta_i = -h \sum_{k=1}^{n} \frac{\partial^2}{
\partial x_k^2}~ + \sum_{k=1}^{n} c_k~ (a^{*k} a^{k}-a^k
a^{*k}) +h^{-1}\sum_{k=1}^{n} c^2_k x^2_k
\end{equation*}
and
\begin{equation*}
\Delta_j = -h \sum_{k=1}^{n-1} \frac{\partial^2} {
\partial x_k^2} + \sum_{k=1}^{n-1} d_k~ (a^{*k}
a^{k}-a^k a^{*k}) +h^{-1}\sum_{k=1}^{n-1} d^2_k x^2_k,
\end{equation*}
where $a^k=(dx^{k})^*$ and $a^{*k}=(a^k)^*$ are the fermionic creation and annihilation operators.
The spectrum of the model operator $\Delta_{\textrm{mod}}$ is the
union of the spectra of the operators $\oplus_{i=1}^{N_0}
\Delta_{i}$ and $\oplus_{j=1}^{N} \Delta_{j}$.\\

The spectra of the operators $\Delta_{i}$ and $\Delta_{j}$ are the
same as in the case of manifolds without boundary (see \cite{S2}).\\

The spectrum of $\Delta_{i}$ is
\begin{eqnarray}\label{E:spec(Deltai)}
\{\sum_{l=1}^{n}(2k_l+1)c_l +(c_{l_1}+ ... +
c_{l_k})-(c_{l_{k+1}}+ ... + c_{l_n})\}
\end{eqnarray}
where $k_l \in \{0,1,2,...\}$,  $l_1<...<l_k$, $ l_{k+1}< ... <
l_n$, $\{l_1 , ... , l_k\} \cup
\{l_{k+1} , ... , l_n\} = \{1 , ... , n\}$ (Corollary 2.22 in \cite{S2}).\\

The spectrum of $\Delta_{j}$ is
\begin{eqnarray}\label{E:spec(Deltaj)}
\{\sum_{l=1}^{n-1}(2k_l+1)d_l +(d_{l_1}+ ... +
d_{l_k})-(d_{l_{k+1}}+ ... + d_{l_{n-1}})\}
\end{eqnarray}
where $k_l \in \{0,1,2,...\}$,  $ l_{k+1}< ... < l_{n-1}$, $\{l_1
, ... , l_k\}
\cup \{l_{k+1} , ... , l_{n-1}\} = \{1 , ... , n-1\}$ (Corollary 2.22 in \cite{S2}).\\

The spectrum of the model operator is the union of
\eqref{E:spec(Deltai)} and \eqref{E:spec(Deltaj)}
over $i=1,...,N_0$ and $j=1,...,N$ respectively.\\

\section{Appendix}\label{S:appendix}

\subsection{Localization theorem for spectral projections}
The goal of this Section is to prove
Proposition~\ref{p:localization} below, which we need for the
proof of Theorem~\ref{t:equivalence}.\\

Let $Q$ be a closed bounded below quadratic form on a Hilbert
space ${\mathcal H}$ with the domain $D(Q)$ which is assumed to be
dense in ${\mathcal H}$. Let $A$ be the self-adjoint
operator corresponding to the quadratic form.\\

Let us take $\lambda_{0}\leq 0$ such that
\begin{equation}\label{e:03}
Q(\omega) \geq \lambda_{0} \|\omega\|^2, \quad \omega \in D(Q).
\end{equation}\\

Let $J$ be a self-adjoint bounded operator in ${\mathcal H}$ that
maps the domain of $Q$ into itself, $J: D (Q)\to D (Q)$. We assume
that $0\leq J\leq {\rm id}_{{\mathcal H}}$. Introduce a
self-adjoint positive bounded operator $J'$ in ${\mathcal H}$ by
the formula $J^2+(J')^2={\rm id}_{{\mathcal H}}$. We assume that
$J'$ maps the domain of $Q$ into itself, the quadratic forms
$Q(\omega)-(Q(J \omega) + Q(J' \omega))$ are bounded i.e. there
exists $\gamma \geq 0$ such that
\begin{equation}\label{e:05}
Q(J \omega) + Q(J' \omega)-Q(\omega) \leq \gamma \|\omega\|^2,
\quad \omega \in D(Q).
\end{equation}\\

Finally, we assume that
\begin{equation}\label{e:04}
Q(J'\omega)\geq \alpha \|J'\omega\|^2,\quad \omega\in D(Q),
\end{equation}
for some $\alpha>0$.\\

Denote by $E(\lambda)$ the spectral projection of the operator
$A$, corresponding to the semi-axis $(-\infty,\lambda]$. We have
\begin{equation}\label{e:01}
Q(E(\lambda)\omega)\leq \lambda\|E(\lambda)\omega\|^2, \quad
\omega\in {\mathcal H}.
\end{equation}\\

\begin{proposition}\label{p:localization}
If $\alpha>\lambda+\gamma$, then we have the following estimate
\begin{equation}\label{e:06}
\|JE(\lambda)\omega\|^2\geq
\frac{\alpha-\lambda-\gamma}{\alpha-\lambda_0}\|E(\lambda)\omega\|^2,
\quad \omega\in {\mathcal H}.
\end{equation}
\end{proposition}
\smallskip
\begin{remark}
Note that in the case $\lambda<\lambda_0$ the statement is
trivial. In the opposite case $\lambda\geq\lambda_0$, since
$\alpha>\lambda+\gamma$ and $\gamma\geq 0$, the coefficient in the
right-hand side of the formula (\ref{e:06}) satisfies the estimate
\[
0<\frac{\alpha-\lambda-\gamma}{\alpha-\lambda_0}\leq 1.
\]
\end{remark}

\bigskip
\begin{proof}(of Proposition ~\ref{p:localization})  Combining \eqref{e:03}, \eqref{e:05} ,\eqref{e:04}
and \eqref{e:01} we get
\begin{align*}
\|J'E(\lambda)\omega\|^2 &\leq
\frac{1}{\alpha}Q(J'E(\lambda)\omega)\\ &\leq
\frac{1}{\alpha}\left(Q(E(\lambda)\omega)-Q(J
E(\lambda)\omega) + \gamma\|E(\lambda)\omega\|^2\right)\\
&\leq \frac{1}{\alpha}\left((\lambda+\gamma)\|E(\lambda)\omega\|^2
-\lambda_0\|J E(\lambda)\omega\|^2 \right).
\end{align*}
Hence, we have
\begin{equation*}
\|JE(\lambda)\omega\|^2=\|E(\lambda)\omega\|^2-\|J'E(\lambda)\omega\|^2
\geq
\left(1-\frac{\lambda+\gamma}{\alpha}\right)\|E(\lambda)\omega\|^2
+\frac{\lambda_0}{\alpha}\|JE(\lambda)\omega\|^2,
\end{equation*}
that immediately implies the required estimate \eqref{e:06}.
\end{proof}
\bigskip

\begin{corollary}
If $\alpha>\lambda+\gamma$, then we have the following estimate:
\begin{equation}\label{e:010}
\|J'E(\lambda)\omega\|^2\leq
\frac{\lambda+\gamma-\lambda_0}{\alpha-\lambda_0}\|E(\lambda)\omega\|^2,
\quad \omega \in {\mathcal H}.
\end{equation}
\end{corollary}
\bigskip
\begin{proof} This follows immediately from the equality
$\|J\omega\|^2+\|J'\omega\|^2=\|\omega\|^2$ for any $\omega\in
{\mathcal H}$.
\end{proof}
\bigskip
\begin{corollary}
If $\alpha>\lambda+\gamma$, then we have the following estimate
\begin{equation}\label{e:011}
Q(JE(\lambda)\omega)\leq
\left(\lambda+\gamma-\lambda_0\frac{\lambda+\gamma-\lambda_0}{\alpha-\lambda_0}
\right)\|E(\lambda)\omega\|^2,\quad \omega\in {\mathcal H}.
\end{equation}
\end{corollary}

\bigskip
\begin{proof} From \eqref{e:05} ,\eqref{e:04}, \eqref{e:01}
and \eqref{e:010} we get
\begin{align*}
Q(J E(\lambda)\omega)&\leq Q(E(\lambda)\omega)-Q(J'E(\lambda)\omega)+ \gamma\|E(\lambda)\omega\|^2\\
&\leq
\left((\lambda+\gamma)\|E(\lambda)\omega\|^2-\lambda_0\|J'E(\lambda)\omega\|^2
\right)\\
&\leq
\left(\lambda+\gamma-\lambda_0\frac{\lambda+\gamma-\lambda_0}{\alpha-\lambda_0}
\right)\|E(\lambda)\omega\|^2
\end{align*}
as desired.
\end{proof}
\bigskip
\subsection{Proof of Theorem~\ref{t:equivalence}} In this Section,
we will use the notation of Section~\ref{s:abstract-equiv}. We
start with the following

\begin{proposition}\label{p:closedimage}
If $\alpha_1>\lambda_1+\gamma_1$ and $$ \lambda_2>\rho\left[
\beta_1 \left(\lambda_1+\gamma_1+
\frac{(\lambda_1+\gamma_1-\lambda_{01})^2}{\alpha_1-\lambda_1-\gamma_1}\right)+
\varepsilon_1\right],
$$ then there exists $\varepsilon_0 >0$
such that
\begin{equation*}
\|E_2(\lambda_2)i_2Jp_1E_1(\lambda_1)\omega\|_2^2\geq
\varepsilon_0 \|E_1(\lambda_1)\omega\|_1^2, \quad \omega\in
{\mathcal H}_1.
\end{equation*}
\end{proposition}

\bigskip
\begin{proof} Applying \eqref{e:16} to a function
$p_1E_1(\lambda_1)\omega,~ \omega\in {\mathcal H}_1$ and taking
into account that $J_1=i_1Jp_1$, we get
\begin{equation}\label{e:7}
Q_2(i_2Jp_1E_1(\lambda_1)\omega)\leq \beta_1
Q_1(J_1E_1(\lambda_1)\omega) +\varepsilon_1
\|J_1E_1(\lambda_1)\omega\|_1^2.
\end{equation}

Clearly, for any $\lambda$ and $l=1,2$ we have the estimate
\begin{equation}
\label{e:4} Q_l(({\rm id}_{{\mathcal H}_l}-E_l(\lambda))\omega)
\geq \lambda\|({\rm id}_{{\mathcal
H}_l}-E_l(\lambda))\omega\|_l^2, \quad \omega\in D (Q_l).
\end{equation}

By \eqref{e:4}, \eqref{e:6} and \eqref{e:rho}, it follows that
\begin{align}\label{e:21}
Q_2(&i_2Jp_1E_1(\lambda_1)\omega)\\
&=Q_2(E_2(\lambda_2)i_2Jp_1E_1(\lambda_1)\omega) +Q_2( ({\rm
id}_{{\mathcal H}_2}-E_2(\lambda_2))i_2Jp_1E_1(\lambda_1)\omega)\notag\\
&\geq
\lambda_{02}\|E_2(\lambda_2)i_2Jp_1E_1(\lambda_1)\omega\|_2^2+
\lambda_2\|({\rm id}_{{\mathcal
H}_2}-E_2(\lambda_2))i_2Jp_1E_1(\lambda_1)\omega\|_2^2\notag\\
&= \lambda_2\|i_2Jp_1E_1(\lambda_1)\omega\|_2^2 -
(\lambda_2-\lambda_{02})\|E_2(\lambda_2)i_2Jp_1E_1(\lambda_1)\omega\|_2^2\notag\\
&\geq \lambda_2\rho^{-1}\|J_1E_1(\lambda_1)\omega\|_1^2 -
(\lambda_2-\lambda_{02})\|E_2(\lambda_2)i_2Jp_1E_1(\lambda_1)\omega\|_2^2.\notag
\end{align}

On the other side, by \eqref{e:7}, \eqref{e:011}, we have
\begin{align}\label{e:22}
Q_2(&i_2Jp_1E_1(\lambda_1)\omega) \leq \beta_1
Q_1(J_1E_1(\lambda_1)\omega) +\varepsilon_1
\|J_1E_1(\lambda_1)\omega\|_1^2\\
&\leq \beta_1 \left(\lambda_1+\gamma_1-\lambda_{01}
\frac{\lambda_1+\gamma_1-\lambda_{01}}{\alpha_1-\lambda_{01}}\right)
\|E_1(\lambda_1)\omega\|_1^2 +\varepsilon_1
\|J_1E_1(\lambda_1)\omega\|_1^2.\notag
\end{align}

Combining \eqref{e:21} and \eqref{e:22}, we get
\begin{multline*}
\lambda_2\rho^{-1}\|J_1E_1(\lambda_1)\omega\|_1^2 -
(\lambda_2-\lambda_{02})\|E_2(\lambda_2)i_2Jp_1E_1(\lambda_1)\omega\|_2^2\\
\leq \beta_1 \left(\lambda_1+\gamma_1-\lambda_{01}
\frac{\lambda_1+\gamma_1-\lambda_{01}}{\alpha_1-\lambda_{01}}\right)
\|E_1(\lambda_1)\omega\|_1^2 +\varepsilon_1
\|J_1E_1(\lambda_1)\omega\|_1^2 ,
\end{multline*}
that implies, due to \eqref{e:06},
\begin{align*}
\|E_2(\lambda_2)i_2Jp_1E_1(\lambda_1)\omega\|_2^2 \geq&
\frac{1}{\lambda_2-\lambda_{02}}
\left[(\lambda_2\rho^{-1}-\varepsilon_1)\|J_1E_1(\lambda_1)\omega\|_1^2\right.\\
&- \left. \beta_1 \left(\lambda_1+\gamma_1-\lambda_{01}
\frac{\lambda_1+\gamma_1-\lambda_{01}}{\alpha_1-\lambda_{01}}\right)
\|E_1(\lambda_1)\omega\|_1^2\right] \\ \geq&
\frac{1}{\lambda_2-\lambda_{02}}
\left[(\lambda_2\rho^{-1}-\varepsilon_1)
\frac{\alpha_1-\lambda_1-\gamma_1}{\alpha_1-\lambda_{01}} \right.\\
&- \left. \beta_1 \left(\lambda_1+\gamma_1-\lambda_{01}
\frac{\lambda_1+\gamma_1-\lambda_{01}}{\alpha_1-\lambda_{01}}\right)\right]
\|E_1(\lambda_1)\omega\|_1^2
\end{align*}
as desired.
\end{proof}
\bigskip
\begin{remark}\label{r:closedimage} Note that we only used estimate~(\ref{e:16}) (but
not (\ref{e:A1A2})) in the proof of
Proposition~\ref{p:closedimage}. By using \ref{e:A1A2} we can get
\begin{equation*}
\|E_1(\lambda_1)i_1Jp_2E_2(\lambda_2)\omega\|_1^2\geq
\varepsilon_0 \|E_2(\lambda_2)\omega\|_2^2, \quad \omega\in
{\mathcal H}_2.
\end{equation*}
\end{remark}

\bigskip
\begin{proof}(of Theorem~\ref{t:equivalence}) As above, we
will use the notation of Section~\ref{s:abstract-equiv}. Take
arbitrary $\lambda_1\in (a_1,b_1)$ and $\lambda_2\in (a_2,b_2)$.\\

Since $$ \lambda_2> a_2=\rho\left[ \beta_1 \left(a_1+\gamma_1+
\frac{(a_1+\gamma_1-\lambda_{01})^2}{\alpha_1-a_1-\gamma_1}\right)+
\varepsilon_1\right] $$ and (see Remark \ref{R:thesame})
\begin{align*}
b_1&=\rho\left[ \beta_2 \left(b_2+\gamma_2+
\frac{(b_2+\gamma_2-\lambda_{02})^2}{\alpha_2-b_2-\gamma_2}\right)+
\varepsilon_2\right]\\ & >\rho\left[ \beta_2
\left(\lambda_2+\gamma_2+
\frac{(\lambda_2+\gamma_2-\lambda_{02})^2}{\alpha_2-\lambda_2-\gamma_2}\right)+
\varepsilon_2\right],
\end{align*}
it follows from Proposition~\ref{p:closedimage} that the map $$
E_2(\lambda_2)i_2Jp_1E_1(\lambda_1) =
E_2(\lambda_2)i_2Jp_1E_1(a_1+0) : {\rm Im} E_1(\lambda_1)\to {\rm
Im} E_2(\lambda_2) $$ is injective and from
Remark~\ref{r:closedimage} the map
\begin{eqnarray*}
(E_2(\lambda_2)i_2Jp_1E_1(\lambda_1))^* &=&
E_1(\lambda_1)i_1Jp_2E_2(\lambda_2) \\&=&
E_1(b_1-0)i_1Jp_2E_2(\lambda_2): {\rm Im} E_2(\lambda_2)\to {\rm
Im} E_1(\lambda_1)
\end{eqnarray*}
is injective. Hence, the map
$E_2(\lambda_2)i_2Jp_1E_1(\lambda_1):{\rm Im} E_1(\lambda_1)\to
{\rm Im} E_2(\lambda_2)$ is bijective.\\

Therefore ${\rm dim} ({\rm Im} E_2(\lambda_2))={\rm dim} ({\rm Im}
E_1(\lambda_1))$. Since the spectrum of $A_1$ does not intersect
with $(a_1, b_1)$, ${\rm dim} ({\rm Im} E_1(\lambda_1))$ is
constant for any $\lambda_1 \in (a_1, b_1)$. Therefore ${\rm dim}
({\rm Im} E_2(\lambda_2))$ is constant for any $\lambda_2 \in
(a_2,b_2)$. This implies that the interval $(a_2,b_2)$ does not
intersect with the spectrum of $A_2$.
\end{proof}

\begin{remark}
If  the spectral projections $E_l(\lambda),~ l=1,2$ have finite
rank for all $\lambda $, then we do not need the condition
$(i_2Jp_1)^*=i_1Jp_2$. In this case the part
$(E_2(\lambda_2)i_2Jp_1E_1(\lambda_1))^* =
E_1(\lambda_1)i_1Jp_2E_2(\lambda_2)$ in the proof is not necessary
to conclude ${\rm dim} ({\rm Im} E_2(\lambda_2))={\rm dim} ({\rm
Im} E_1(\lambda_1))$. We can consider the maps $$
E_2(\lambda_2)i_2Jp_1E_1(\lambda_1) =
E_2(\lambda_2)i_2Jp_1E_1(a_1+0) : {\rm Im} E_1(\lambda_1)\to {\rm
Im} E_2(\lambda_2) $$ and
$$E_1(\lambda_1)i_1Jp_2E_2(\lambda_2) =
E_1(b_1-0)i_1Jp_2E_2(\lambda_2): {\rm Im} E_2(\lambda_2)\to {\rm
Im} E_1(\lambda_1). $$ These maps are injective by
Proposition~\ref{p:closedimage} and Remark \ref{r:closedimage}.
Therefore,\\ ${\rm dim} ({\rm Im} E_2(\lambda_2))={\rm dim} ({\rm
Im} E_1(\lambda_1))$ because ${\rm Im}E_1(\lambda_1)$ and ${\rm
Im} E_2(\lambda_2)$ are finite dimensional.
\end{remark}
\bigskip
Denote the spectral projection of the operator $A_l$ on the
interval $(a, b)$ as $E_l(a,b)$ for $l=1,2$.

\begin{corollary}
For any $\lambda \in {\rm spec}\,(A_1)$, there is a positive
number $\epsilon$ which is of order $h^\kappa$ such that the
spaces $E_1((\lambda-\epsilon , \lambda + \epsilon) )$ and
$E_2((\lambda-\epsilon , \lambda + \epsilon) )$ are isomorphic.
\end{corollary}
\bigskip

\end{document}